\documentclass[12pt]{article}
\usepackage{latexsym}
\usepackage{multirow}
\usepackage{graphicx}
\usepackage{float}

\usepackage{authblk}

\usepackage{multirow}
\usepackage{kantlipsum}
\usepackage{setspace}

\usepackage{xcolor, color}
\usepackage{bbm}  
\def\q{\hfill\rule{1ex}{1ex}}
\def\0{\emptyset}

\usepackage{indentfirst}
\usepackage{amsmath}

\def\0{\emptyset}

\def\q{\hfill\rule{1ex}{1ex}}

\usepackage{mathrsfs}
\usepackage{amssymb}
\usepackage{listings}
\usepackage{amsmath}
\usepackage{pict2e}
\usepackage{amsfonts}
\usepackage{graphicx}
\usepackage[small]{caption}
\usepackage{subfigure}
\usepackage{mathtools}
\usepackage{multirow}
\usepackage{tikz}
\newtheorem{theorem}{Theorem}[section]
\newtheorem{corollary}[theorem]{Corollary}

\newtheorem{prob}[theorem]{Problem}	\newtheorem{lemma}[theorem]{Lemma}
\newtheorem{prop}[theorem]{Proposition}
\newtheorem{proposition}[theorem]{Proposition}

\newenvironment{proof}{\noindent {\bf Proof.}}{\rule{2mm}{2mm}\par\medskip}

\newcommand\red[1] {{\color{red} #1}}

\newcommand\blue[1] {{\bf \color{blue} #1}}

\newcommand\equ[2]
{
	\begin{equation} \label{#1}
		#2
	\end{equation} 
}

\newcommand\eqn[2]
{
	\begin{eqnarray} \label{#1}
		#2
	\end{eqnarray} 
}

\newcommand\eqnn[1]
{
	\begin{eqnarray*} 
		#1
	\end{eqnarray*} 
}

\newcounter{countclaim}

\begin{document}
	
	\title{A neighborhood union 
		condition 
		for the existence of 
		a spanning tree without 
		degree $2$ vertices\thanks{This research was partially supported by NSFC (No. 12371345, 12471325 and 12371340). }
	}
	
	\author[1,2]{\small Yibo Li\thanks{Email: liyibo@stu.hubu.edu.cn}}
	
	\author[3]{\small Fengming Dong\thanks{Email: fengming.dong@nie.edu.sg and donggraph@163.com}}
	
	\author[4]{\small 
		Xiaolan Hu\thanks{Email: xlhu@ccnu.edu.cn}}
	
	\author[1,2]{\small Huiqing Liu\thanks{Corresponding author. Email: hqliu@hubu.edu.cn}}

	\affil[1]{\footnotesize 
		Hubei Key Laboratory of Applied Mathematics, Faculty of Mathematics and Statistics, 
		
		Hubei University, Wuhan 430062, China}
	
	\affil[2]{\footnotesize 
		Key Laboratory of Intelligent Sensing System and Security (Hubei University), Ministry of Education}
	
	\affil[3]{\footnotesize 
		National Institute of Education,
		Nanyang Technological University, 
		Singapore}

	\affil[4]{\footnotesize 
		School of Mathematics and Statistics \& Hubei Key Laboratory of Mathematical Sciences, 
		
		Central China Normal University, Wuhan 430079, PR China
	}

	\date{}
	\maketitle

	\begin{abstract}
		For a connected graph $G$, 
		a spanning tree $T$ of $G$ is called a homeomorphically irreducible spanning tree (HIST)
		if $T$ has no vertices of degree $2$.
		In this paper, we show that 
		if $G$ is a graph of order $n\ge 270$  and 
		$|N(u)\cup N(v)|\geq\frac{n-1}{2}$ 
		holds for every pair of nonadjacent vertices $u$ and $v$ in $G$, then
		$G$ has a HIST, unless $G$ 
		belongs to three
		exceptional families of graphs 
		or $G$ has a
		cut-vertex of degree $2$.
		This result improves the latest conclusion, due to Ito and Tsuchiya, 
		that a HIST in $G$ can be guaranteed 
		if $d(u)+d(v)\geq n-1$ holds 
		for every pair of nonadjacent vertices $u$ and $v$ in $G$.
		

	\end{abstract}
	


	\section{Introduction}
	We consider only simple graphs in this article. 
	For any graph $G$, let $V(G)$ and $E(G)$ denote the set of vertices 
	and the set of edges in $G$. 
	For $v\in V(G)$, let $N_G(v)$ be 
	the {\em set of neighbors} of $v$ in $G$,  $N_G[v]:=N_G(v)\cup\{v\}$, and 
	$d_G(v):=|N_G(v)|$ denote 
	the {\em degree} of $v$ in $G$,
	where the index $G$
	will be omitted 
	 if there is no risk of confusion.
	A {\em pendant vertex} of $G$ is a vertex of degree $1$ in $G$.
	For any non-empty subset $S$ of $V(G)$, let $G[S]$ denote the subgraph of $G$ induced by $S$,
	and write $N_S(v)$ and $d_S(v)$
	for $N_G(v)\cap S$ and $|N_G(v)\cap S|$, respectively, 
	for each  $v\in V(G)$.

	We denote by $G-S$ the subgraph of $G$ induced by $V(G)\backslash S$. For a proper subgraph $H$ of $G$ and subset $S$ of $V(G)$, let $H+S$ be the subgraph of $G$ induced by $V(H)\cup S$. If $S=\{v\}$, then we simplify $G-\{v\}$ (or resp. $H+\{v\}$) to $G-v$ (or resp. $H+v$). A subset $C$ of $V(G)$ is called a {\em clique} if every two vertices in $C$ are adjacent. By $K_n$, we denote the {\em complete graph} of order $n$. For a subgraph $H$ of $G$, we consider it as both a subgraph and a vertex set of $G$. 
	
	Given two disjoint vertex sets $X$ and $Y$ of $G$, let $E(X,Y)$ be the set of all edges with one end in $X$ and one end in $Y$. If $X=\{x\}$, we simply write $E(x,Y)$ for $E(X,Y)$. When $Y=V\backslash X$, the set $E(X,Y)$ is called the {\em edge cut} of $G$ associated with $X$, and is denoted by $\partial(X)$. A {\em bond} of $G$ 
	is a minimal nonempty edge cut of $G$. 
	For a connected graph $G$, let
	\begin{eqnarray*}
		\delta(G)&:=&\mbox{min}\{d(u):u\in V(G)\},\\
		\sigma(G)&:=&\mbox{min}\{d(u)+d(v):uv\notin E(G),u\neq v\},\\
		NC(G)&:=&\mbox{min}\{|N(u)\cup N(v)|:uv\notin E(G),u\neq v\}.
	\end{eqnarray*} 
	Clearly, $\delta(G)\leq NC(G)\leq\sigma(G)$.
	
	For a connected graph $G$, a spanning tree $T$
	of $G$ is called a {\em homeomorphically irreducible spanning tree} (HIST) of $G$ if 
	$T$ has no vertices of degree 2. 
	Similar to the study of Hamiltonian graphs,
	the existence of a HIST has been studied in relation with $\delta(G)$, 
	$\sigma(G)$, $NC(G)$, 
	or other parameters  (see \cite{ABHT90,    ChRS12, Chsh13, FuTs20, ItTs22} for example). 
	Albertson, Berman, Hutchinson and Thomassen \cite{ABHT90} first found the condition $\delta(G)\ge 4\sqrt{2n}$ for 
	the existence of a HIST in $G$. 
	This condition was recently replaced by a weaker one 
	by Furuya, Saito and Tsuchiya~\cite{FuST24}.
	
	\begin{theorem}[\cite{FuST24}] \label{Thm1.2}
		Let $G$ be a connected graph of order $n$.
		If  
		$\delta(G)\geq4\sqrt{n}$, 
		then $G$ has a HIST.
	\end{theorem}

	
	

	In 2022, Ito and Tsuchiya \cite{ItTs22}
	found a condition on $\sigma(G)$ for
	the existence of a HIST. 
	
	\begin{theorem}[\cite{ItTs22}]	\label{Thm1.3}
		Let $G$ be a graph of order $n\geq8$.
			If 
		$\sigma(G)\geq n-1$, 
		then $G$ has a HIST.
	\end{theorem}

	Broersma, Heuvel and Veldman
	\cite{BrDV93} showed that 
	under the condition $NC(G)\geq\frac{n}{2}$, $G$ is either Hamiltonian, or $G$ is the Petersen graph or $G$ belongs to three special families of graphs. 
	Motivated by the result in 
	\cite{BrDV93}, 
	we study the 
	existence of a HIST
	in a graph $G$ with a condition of 
	$NC(G)$ and obtain the following conclusion.

We are now going to introduce three graphs $H_1, H_2$ and $H_3$ of order $n$ shown in Figure~\ref{f1}, where $n\ge 5$ is odd,
before presenting our main result in this article. 
$H_1$ is a graph  which contains a cut-vertex $v$ of degree $2$ such that 
$H_1-v$ has exactly two 
components each of which is isomorphic to 
$K_{\frac{n-1}{2}}$. 
Suppose that $G_0, G_1$ and $G_2$ are vertex-disjoint graphs, where 
$G_0$ is  
isomorphic to $K_3$ with 
vertex set 
$\{v_1,v_2,v_3\}$, and for each $i\in [2]$, $G_i$ is isomorphic to  $K_{\frac{n-3}2}$.
$H_2$ is the graph obtained from 
$G_0, G_1$ and $G_2$
by adding an edge 
joining $v_i$ to a vertex 
$u_i$ in $G_i$ for each $i\in [2]$, while $H_3$ is the graph obtained from 
$H_2$
by adding an edge 
joining $v_3$ to a vertex 
$u_3$ in $G_2$. 
It is possible that $u_2$ and $u_3$ are the same vertex.

	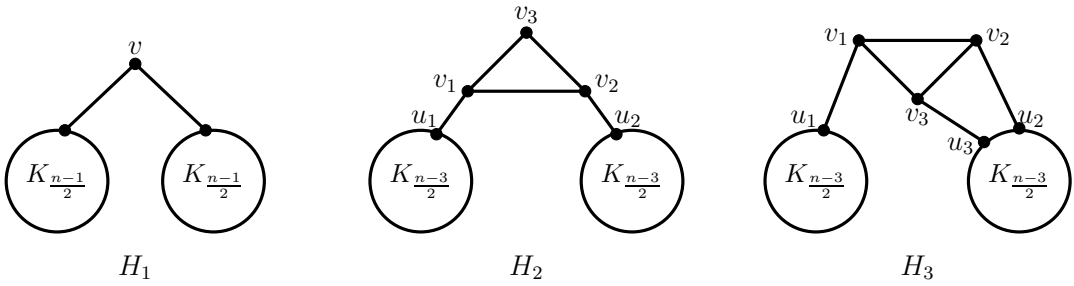
\begin{figure}[H]
		\begin{center}  
			\begin{tikzpicture}[scale=0.52]
				\tikzstyle{every node}=[font=\normalsize,scale=0.9]
				\draw (-5,9.2)node[above=0pt]{$v$};
				\filldraw (-5,9.2)circle(0.8ex);
				\filldraw (-6.8,7.5)circle(0.8ex);
				\filldraw (-3.2,7.5)circle(0.8ex);
				\draw[very thick] (-7,6.2)circle(1.3);
				\draw[very thick] (-3,6.2)circle(1.3);
				\draw[very thick] (-5,9.2)--(-6.8,7.5);
				\draw[very thick] (-5,9.2)--(-3.2,7.5);
				\draw (-7,6.2) node{\small$K_{\frac{n-1}{2}}$};
				\draw (-3,6.2) node{\small$K_{\frac{n-1}{2}}$};
				\draw (-5,4) node[]{$H_1$};
				~~~~~~~~~~~~~~~~~~~~~~~~~~~~~~~~~~~~~~~~~~~~~~~~~~
				\filldraw (5,10)circle(0.8ex);
				\draw (5,10)node[above=0pt]{$v_3$};
				\filldraw (6.5,8.5)circle(0.8ex);
				\draw (6.5,8.7)node[right=0pt]{$v_2$};
				\filldraw (3.5,8.5)circle(0.8ex);
				\draw (3.5,8.7)node[left=0pt]{$v_1$};
				\filldraw (2.7,7.4)circle(0.8ex);
				\draw (3,7.8)node[left=0pt]{$u_1$};
				\filldraw (7.3,7.4)circle(0.8ex);
				\draw (7,7.8)node[right=0pt]{$u_2$};
				\draw[very thick] (5,10)--(6.5,8.5);
				\draw[very thick] (5,10)--(3.5,8.5);
				\draw[very thick] (6.5,8.5)--(3.5,8.5);
				\draw[very thick] (6.5,8.5)--(7.3,7.4);
				\draw[very thick] (3.5,8.5)--(2.7,7.4);
				\draw[very thick] (2.3,6.2)circle(1.3);
				\draw[very thick] (7.7,6.2)circle(1.3);
				\draw (2.3,6.2) node{\small$K_{\frac{n-3}{2}}$};
				\draw (7.7,6.2) node{\small$K_{\frac{n-3}{2}}$};
				\draw (5,4) node[]{$H_2$};
				~~~~~~~~~~~~~~~~~~~~~~~~~~~~~~~~~~~~~~~~~~~~~~~~~~~	
				\filldraw (15,8.3)circle(0.8ex);
				\draw (15,8.3)node[below=0pt]{$v_3$};
				\filldraw (13.5,9.8)circle(0.8ex);
				\draw (13.5,9.9)node[left=0pt]{$v_1$};
				\filldraw (16.5,9.8)circle(0.8ex);
				\draw (16.5,9.9)node[right=0pt]{$v_2$};
				\filldraw (12.6,7.5)circle(0.8ex);
				\draw (12.7,7.8)node[left=0pt]{$u_1$};
				\filldraw (17.6,7.55)circle(0.8ex);
				\draw (17.3,7.8)node[right=0pt]{$u_2$};
				\filldraw (16.7,7.2)circle(0.8ex);
				\draw (16.7,7.1)node[left=0pt]{$u_3$};
				\draw[very thick] (15,8.3)--(16.7,7.2);
				\draw[very thick] (16.5,9.8)--(17.6,7.5);
				\draw[very thick] (13.5,9.8)--(16.5,9.8);
				\draw[very thick] (13.5,9.8)--(15,8.3);
				\draw[very thick] (16.5,9.8)--(15,8.3);
				\draw[very thick] (13.5,9.8)--(12.6,7.5);		
				\draw[very thick] (12.4,6.2)circle(1.3);
				\draw[very thick] (17.6,6.2)circle(1.3);
				\draw (12.4,6.2) node{\small$K_{\frac{n-3}{2}}$};
				\draw (17.6,6.2) node{\small$K_{\frac{n-3}{2}}$};
				\draw (15,4) node[]{$H_3$};
				~~~~~~~~~~~~~~~~~~~~~~~~~~~~~~~~~~~~~~~~~~~~~~~~~~~~		
			
			\end{tikzpicture}
		\end{center}
		
		\caption{Graphs $H_1$, $H_2$
			and $H_3$ of order $n$, where $n\ge 5$ is odd.  $u_2$ and $u_3$ in $H_3$  may coincide}
		
		\label{f1}
	\end{figure}

	For any integer $k$, 
	let $[k]$ denote the set $\{1,2,\dots,k\}$.
	
	\begin{theorem}\label{Thm1.5}
		Let $G$ be a connected graph of order $n$, where $n\geq270$. If 
		\begin{equation} \label{Th1.5-e1} 
			NC(G)\geq\frac{n-1}{2},
		\end{equation} 
		then $G$ has a HIST
		if and only if $G$ is not isomorphic to any $H_i$, $i\in [3]$,
		shown in Figure~\ref{f1}, 
		nor a graph with  a pendant vertex which is adjacent to a vertex of degree $2$. 
		\end{theorem}

	\medskip
	\noindent{\bf Remarks.}
	(i)	 It can be proved easily that 
	$\sigma(G)\ge n-1$ implies that 
	$NC(G)\ge \frac{n-1}2$. Thus, 
	Theorem~\ref{Thm1.5} is a generalization of 
	Theorem~\ref{Thm1.3}. 

	(ii) It will be proved in Section 2 that each $H_i$, where $i\in [3]$,  shown in Figure~\ref{f1}, does not 
	have a HIST. 
	Any graph with 
	a pendant vertex which is adjacent to a vertex of degree $2$
	also has no HISTs. 
	The details are given in  Corollary~\ref{cor1}.
	Thus, Theorem~\ref{Thm1.5} actually characterizes graphs of order at least $270$ which have HISTs.

	It can be verified that
	$NC(H_i)=\frac{n-1}2$ for each 
	$i\in [3]$.
	Thus, the following conclusion follows 
	directly from Theorem~\ref{Thm1.5}.
	
	\begin{corollary}
		Let $G$ be a connected graph 
		of order $n\geq270$. If $NC(G)\geq\frac{n}{2}$, 
		then $G$ has a HIST
		if and only if $G$ does not have a pendant vertex which is adjacent to a vertex of degree $2$.
	\end{corollary}

	\section{Preliminaries}

	In this section, we first explain 
	why the conclusion of 
	Theorem~\ref{Thm1.5} excludes 
	all the four graphs $H_i$'s shown 
	in Figure~\ref{f1} and 
	any graph with a pendant vertex which is adjacent to a vertex of degree 2.
	Actually, any graph with a cut-vertex of degree 2 has no HISTs, as stated below. 
	
	\begin{proposition} \label{Prop2.2}
		Any connected graph with a 
		cut-vertex of degree $2$ 
		has no HISTs. 
	\end{proposition}

	\begin{proposition} \label{Prop2.3}
		Let $G$ be a connected graph and let $S:=v_1v_2v_3v_1$ a triangle of $G$. 
		If $d(v_i)\leq3$ for each 
		$i\in [3]$ and 
		$G-S$ contains exactly two components, then $G$ has no HISTs.
	\end{proposition}
	
	\begin{proof}
		Let $C_1$ and $C_2$ be the two components of $G-S$.
		Since $G$ is connected, we have $\partial(C_j)\neq\emptyset$ for 
		each $j\in [2]$.
		Assume $|\partial(C_1)|\leq|\partial(C_2)|$. Since $S$ is a triangle and $d(v_i)\leq3$ for all $v_i\in S$, $|\partial(C_1)|+|\partial(C_2)|\leq3$. Thus $|\partial(C_1)|=1$ and $1\leq|\partial(C_2)|\leq2$. 
		We can assume that
		$v_ix_i\in E(G)$ for each $i\in [2]$,
		where $x_i\in C_i$. 
		Then $E(S,C_1)=\{v_1x_1\}$  
		is an edge-cut of $G$. 
		
		Suppose $G$ has a HIST $T$. 
		Clearly, $v_1x_1\in E(T)$. 
		Note that $\partial(C_1\cup\{v_1\})=\{v_1v_2,v_1v_3\}$, and hence $E(T)\cap\{v_1v_2,v_1v_3\}\neq\emptyset$. Furthermore, $v_1v_2,v_1v_3\in E(T)$ as $d_T(v_1)\neq2$. Since $T$ is acyclic, $v_2v_3\notin E(T)$, i.e., $1\leq d_T(v_2)\leq2$ and $1\leq d_T(v_3)\leq2$. Notice that $\partial(C_2)\neq\emptyset$, then $E(T)\cap \partial(C_2)\neq\emptyset$, which implies that $d_T(v_2)=2$ or $d_T(v_3)=2$, a contradiction to the definition of HIST. 
	\end{proof}
	
	
	The following conclusion then follows directly from Propositions
	~\ref{Prop2.2} and~\ref{Prop2.3}.
	
	\begin{corollary}\label{cor1}
		If $G$ contains a pendant 
		vertex which is adjacent to some
		vertex of degree $2$, 
		or $G\cong H_i$, where $H_i$
		is a graph shown in Figure~\ref{f1}, 
		then $G$ does not contain HISTs. 
	\end{corollary}

	In the remainder of this section, 
	we present some results that will be used in the following sections.

	\begin{lemma}\label{le2.4}
		Let $G$ be a graph with $X\subset V(G)$ and $z\in V(G)\backslash X$. If $G[X]$ is connected and $1\leq |N_X(z)|<|X|$, 
		then there exists 
		an induced path $zxy$ in $G[X\cup\{z\}]$, where $x,y\in X$.
	\end{lemma}

	For two nonempty vertex sets $X,Y\subseteq V(G)$, an $(X,Y)$-{\em path} is a path in $G$ 
	which starts at a vertex of $X$, ends at a vertex of $Y$, and whose internal vertices belong to
	$V(G)\backslash (X\cup Y)$.
	
	
	\begin{lemma}[\cite{Bondy08}] \label{le2.5}
		Let $G$ be a connected graph. 
		For any two nonempty subsets $X$ and $Y$ of $V(G)$,  there is an $(X,Y)$-path in $G$.
	\end{lemma}
	
	\begin{lemma} [\cite{Bondy08}]
		\label{le2.6}
		Let $G$ be a graph with $n$ vertices. If $\delta(G)>\frac{n-2}{2}$, then $G$ is connected.
	\end{lemma}

	\section{Non-complete graphs $G$ 
		with $NC(G)\ge \frac{n-1}2$} 
	
	In this section, we always assume that 
	$G$ is a connected graph 
	of order $n$ such that $n>\delta(G)+1$ (i.e., $G$ is not complete)
	and $NC(G)\geq\frac{n-1}{2}$
	(i.e.,  the condition of (\ref{Th1.5-e1}) holds). 
	We also assume that $u$ is a vertex in $G$ with 
	$d(u)=\delta(G)$,   $N(u)=\{u_1,u_2,\ldots,u_{\delta(G)}\}$ and $W=V(G)\backslash N[u]$. Since $n>\delta(G)+1$, 
	we have $W\neq\emptyset$ and $E(N(u),W)\neq\emptyset$. 
	We are now going to establish some  conclusions on $G$ which will be applied in the proof of Theorem~\ref{Thm1.5}.

	\begin{lemma}\label{le2.8}
		If $\delta(G)<\frac{n-3}{2}$, 
		then $\{u_i: i\in [\delta(G)], |N_W(u_i)|\leq1\}$  is a clique.
	\end{lemma}

	\begin{proof} 
		Let $S=\{u_i: i\in [ \delta(G)], |N_W(u_i)|\leq1\}$. 
		Suppose the result fails. Then 
		there exist two vertices $u_p, u_q\in S$ such that $u_pu_q\notin E(G)$.
		Note that 
		$$
		N(u_p)\cup N(u_q)\subseteq(N[u]\backslash\{u_p,u_q\})\cup N_W(u_p)\cup N_W(u_q).
		$$ 
		Then 
		$$|N(u_p)\cup N(u_q)|\leq(\delta(G)+1-2)+2
		<\frac{n-3}{2}+1
		=\frac{n-1}{2},$$ a contradiction to the condition of (\ref{Th1.5-e1}).
	\end{proof}
	
	\medskip
	
	Note that for each $w\in W$, $uw\notin E(G)$, implying that
	$|N(u)\cup N(w)|\geq\frac{n-1}{2}$
	by the given condition. 
	Thus, 
	\begin{equation}\label{e4} 
		\forall w\in W: \quad	d_{W}(w)\geq \frac{n-1}{2}-\delta(G).
	\end{equation}

	\begin{lemma} \label{le2.9}
		For any $S\subset W$ with $|S|<\frac{n+5}{4}-\delta(G)$, $G[W\backslash S]$ contains at most two components. 
	\end{lemma} 
	
	\begin{proof}
		Suppose that $G[W\backslash S]$ contains at least three components. Then there is a component $C_0$ of $G[W\backslash S]$ satisfying 
		\begin{equation} \label{e2} 
			|C_0|\leq\frac{|W\backslash S|}{3}
			=
			\frac{n-1-\delta(G)-|S|}{3}.
		\end{equation} 
		Note that for each $x\in C_0$,
		$N(x)\subseteq(C_0\backslash\{x\})\cup S\cup N(u)$, 
		implying that 
		\begin{eqnarray} \label{e5}
			|N(u)\cup N(x)|
			&\leq& |C_0|-1+|S|+\delta(G)
			\ \leq\ 
			\frac{n+2|S|+2\delta(G)-4}{3}
			\nonumber \\
			&<&
			\frac{n+2\cdot\frac{n+5}{4}-4}{3}
			\ =\ \frac{n-1}{2},
		\end{eqnarray} 
		a contradiction to the condition of (\ref{Th1.5-e1}). So $G[W\backslash S]$ contains at most two components.
	\end{proof}

	\begin{lemma}\label{le2.10}
		Assume that $S\subset W$ and $G[W\backslash S]$ contains exactly two components $C_1$ and $C_2$. Then, for each $i\in[2]$, 
		\begin{enumerate}
			\item[(i)] $\frac{n+1}{2}-\delta(G)-|S|\leq|C_i|\leq\frac{n-3}{2}$;
			
			\item[(ii)]  if $S=\emptyset$, then $|N_{N(u)}(x)\cup N_{N(u)}(y)|\geq3$
			holds for any each pair of non-adjacent 
			vertices  $x$ and $y$  in $C_i$, and 
			
			\item[(iii)]  if $n\geq 143$
			and $|S|+\delta(G)< \frac{n+1}4$,
			then $C_i-S_i$ contains a HIST
			for any $S_i\subset C_i$ with $|S_i|\leq2$.
		\end{enumerate}
	\end{lemma}
	
	\begin{proof} (i) 
		Let $x$ be any vertex in $C_i$, where $i\in [2]$. 
		Then by (\ref{e4}), 
		we have
		\equ{e9}
		{d_{C_i}(x)=d_{W\backslash S}(x)\geq d_W(x)-|S|\geq\frac{n-1}{2}
			-\delta(G)-|S|.}
		
		Then 
		$$
		|C_i|\geq d_{C_i}(x)+1\geq\frac{n+1}{2}-\delta(G)-|S|,
		$$ 
		implying that 
		$$
		|C_i|=|W\backslash S|-|C_{3-i}|\leq\left(n-1-\delta(G)-|S|\right)-\left(\frac{n+1}{2}-\delta(G)-|S|\right)=\frac{n-3}{2}.
		$$ 
		Hence (i) holds.
		
		(ii) 
		Suppose that $x$ and $y$ are any non-adjacent vertices in $C_i$ such that  
		$|N_{N(u)}(x)\cup N_{N(u)}(y)|\leq2$.
		Since $S=\emptyset$, we have 
		$$
		N(x)\cup N(y)\subseteq (C_i\backslash\{x,y\})\cup N_{N(u)}(x)\cup N_{N(u)}(y).
		$$  
		It follows that 
	\equ{le2.10-e2}
	{
		|N(x)\cup N(y)|\leq|C_i|-2+|N_{N(u)}(x)\cup N_{N(u)}(y)|\leq|C_i|-2+2
		=|C_i|\leq\frac{n-3}{2}<\frac{n-1}{2},
	}
		a contradiction to the condition of (\ref{Th1.5-e1}). Hence (ii) holds.
		
		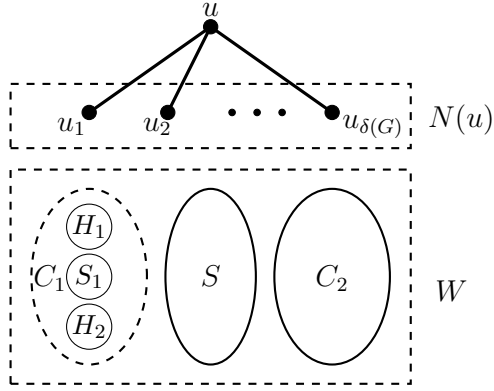
\begin{figure}[H]
			\begin{center}  
				\begin{tikzpicture}[scale=0.95]
					\tikzstyle{every node}=[font=\normalsize,scale=0.95]
					\draw (0,9) node[above=0pt]{$u$};
					\draw (-1.6,7.6)node[left=0pt]{$u_1$};
					\draw (-0.4,7.6)node[left=0pt]{$u_2$};
					\draw (1.7,7.6) node[right=0pt]{$u_{\delta(G)}$};
					\draw (3,5.3) node[right=0pt]{$W$};
					\draw (2.9,7.7) node[right=0pt]{$N(u)$};							
					\filldraw (0,9)circle(0.55ex);
					\filldraw (-1.7,7.8)circle(0.55ex);
					\filldraw (-0.6,7.8)circle(0.55ex);
					\filldraw (1.7,7.8)circle(0.55ex);
					\filldraw (0.3,7.8)circle(0.25ex);
					\filldraw (0.65,7.8)circle(0.25ex);
					\filldraw (1,7.8)circle(0.25ex);
					\draw[] (0,5.5)[thick] ellipse [x radius=18pt, y radius=35pt];
					\draw (0,5.8)[] node[below=0pt]{$S$};
					\draw[] (-1.7,5.5)[thick, dashed] ellipse [x radius=23pt, y radius=35pt];
					\draw (-2.25,5.8)[] node[below=0pt]{$C_1$};
					\draw[] (1.7,5.5)[thick] ellipse [x radius=23pt, y radius=35pt];
					\draw (1.7,5.8)[]node[below=0pt]{$C_2$};
					\draw[] (-1.7,6.2) ellipse [x radius=9pt, y radius=9pt];
					\draw (-1.7,6.2) node[]{\small$H_1$};
					\draw[] (-1.7,5.5) ellipse [x radius=9pt, y radius=9pt];
					\draw (-1.7,5.5) node[]{\small$S_1$};
					\draw[] (-1.7,4.8) ellipse [x radius=9pt, y radius=9pt];
					\draw (-1.7,4.8) node[]{\small$H_2$};
					\draw[very thick] (0,9)--(-1.7,7.8);		
					\draw[very thick] (-0,9)--(-0.6,7.8);		
					\draw[very thick] (0,9)--(1.7,7.8);		
					\draw[thick, dashed] (-2.8,4)rectangle(2.8,7);
					\draw[thick, dashed] (2.8,8.2)rectangle(-2.8,7.3);
				\end{tikzpicture}
			\end{center}	
			\caption{A partition of $V(G)$~~~~~~~~~~~}
			\label{f2}
		\end{figure} 
		
		(iii) We will apply Theorem~\ref{Thm1.2} to prove (iii). 
		We first show that both  $C_i-S_i$ are connected. Without loss of generality, 
		suppose that $C_1-S_1$ is disconnected. 
		Let $H_1$ and $H_2$ be any two components of $C_1-S_1$, as shown in Figure~\ref{f2}. 
		For each $w'\in C_1-S_1$, 
		by (\ref{e4}), we have
		\begin{eqnarray}\label{e60}
			d_{H_j}(w')=d_W(w')-d_S(w')
			-d_{S_1}(w')
		\ge	\frac{n-1}{2}
			-\delta(G)-|S|-|S_1|.
		\end{eqnarray} 
		Assume that  $x\in H_1$ and $y\in H_2$.
		Then, by (\ref{e60}), 
		\equ{e6}
		{
d_{H_1}(x)+d_{H_2}(y)
					\ge
2\left(\frac{n-1}{2}
-\delta(G)-|S|-|S_1|\right)
		=n-1-2\left(\delta(G)+|S|\right)
		-2|S_1|.
	}	
		Since $|S_1|\le 2$ and $\delta(G)+|S|< 
		\frac{n+1}4$, (\ref{e6}) implies that 
		\begin{eqnarray}\label{eq61}
			d_{H_1}(x)+d_{H_2}(y)
			&>&n-1-2\cdot \frac{n+1}4-2-|S_1|
		=\frac{n-3}{2}-2-|S_1|
			\nonumber \\
			&\ge& |C_1|-2-|S_1|
		\	\ge\ |H_1|-1+|H_2|-1,
		\end{eqnarray} 
		a contradiction to the fact that 
		$d_{H_1}(x)\le |H_1|-1$ and 
		$d_{H_2}(y)\le |H_2|-1$, 
		where the second last inequality 
		follows from the result in (i)
		that $|C_1|\le \frac{n-3}2$.
		
		Now we know that 
		$C_i-S_i$ is connected.
		In the following, we will show that 
		$\delta(C_i-S_i)\ge 4\sqrt{|C_i-S_i|}$.
		Let $w\in C_i-S_i$ with $d_{C_i-S_i}(w)=\delta(C_i-S_i)$. Then 
		by (\ref{e9}), 
		\equ{e-n1}
		{
			\delta(C_i-S_i)=d_{C_i-S_i}(w)
			=
			d_{C_i}(w)-d_{S_i}(w)
			\ge \frac{n-1}{2}
			-\delta(G)-|S|-|S_i|.
		}
		
		Since $\delta(G)+|S|<\frac{n+1}4$, by (\ref{e-n1}), we have 
		\equ{e7}
		{\delta(C_i-S_i)
			>\frac{n-1}{2}-\frac{n+1}{4}-|S_i|
			=\frac{n-3}4-|S_i|.
		}
		Since $|S_i|\in \{0,1,2\}$,  it can be verified directly that 
		\eqn{e8}
		{\delta(C_i-S_i)-4\sqrt{\frac{n-3}{2}-|S_i|}
			&> & \frac{n-3}4-|S_i|
			-4\sqrt{\frac{n-3}{2}-|S_i|}
			\nonumber \\
			&\ge& \frac{n-3}4-2
			-4\sqrt{\frac{n-3}{2}-2}\ge 0,
		}
		where the last inequality follows from the condition $n\ge 143$.
		Then, by (\ref{e8}), we have 
		\equ{e10} 
		{
			\delta(C_i-S_i)>4\sqrt{\frac{n-3}{2}-|S_i|}
			\ge 4\sqrt {|C_i-S_i|}.
		}
		Thus, by Theorem~\ref{Thm1.2}, $C_i-S_i$ has a HIST.
	\end{proof}
	
	\medskip

	\section{$1$-quasi-HIT and $2$-quasi-HIT} 
	
	In this section, we still assume that $G$ is a connected graph 
	of order $n$ such that $n>\delta(G)+1$ (i.e., $G$ is not complete)
	and $NC(G)\geq\frac{n-1}{2}$. 
	We also assume that $u$ is a vertex in $G$ with 
	$d(u)=\delta(G)$,   $N(u)=\{u_1,u_2,\ldots,u_{\delta(G)}\}$ and $W=V(G)\backslash N[u]$. 
	
	Any subtree $T$ of $G$ is a {\em HIT} of $G$ if $T$ has no vertices of degree 2. 
	A {\em $1$-quasi-HIT} of $G$ 
	is a subtree $T(v)$ of $G$ 
	such that $v$ is the only vertex of degree $2$ in $T(v)$
	and a {\it $1$-quasi-HIST} of $G$
	is a $1$-quasi-HIT of $G$ which is a spanning tree of $G$.
	Similarly, 
	a {\em $2$-quasi-HIT} of $G$ 
	is a subtree $T(v,w)$ of $G$
	such that $v$ and $w$ are the only 
	vertices of degee $2$ in $T(v,w)$,
	and a $2$-quasi-HIST of $G$
	is a $2$-quasi-HIT of $G$ which is a spanning tree of $G$.

	The two results below follow directly.

	
	\begin{lemma}\label{le2.11}
		Let $T(v)$ be a $1$-quasi-HIT of $G$. If there exists $S\subset V(G)$ such that $S\cap V(T(v))=\{v\}$ and $G[S]$ has a HIST, then $T(v)$ can be extended to a HIT $T'$ of $G$ with  $V(T')=V(T(v))\cup S$. In particular, if $V(T(v))\cup S=V(G)$, then $T'$ is a HIST of $G$.
	\end{lemma}
	
	\begin{lemma}\label{le2.12}
		Let $T(v,w)$ be a $2$-quasi-HIT of $G$. 
		If there exist two vertex-disjoint 
		subsets $S$ and $U$ of $V(G)$ such that $S\cap V(T(v,w))=\{v\}$, $U\cap V(T(v,w))=\{w\}$, and both $G[S]$ and $G[U]$ have HISTs, then $T(v,w)$ can be extended to a HIT $T'$ with $V(T')=V(T(v,w))\cup S\cup U$. In particular, if $V(T(v,w))\cup S\cup U=V(G)$, then $T'$ is a HIST of $G$.
	\end{lemma}

	\begin{lemma} \label{le2.13}
		Assume that $n\ge 259$ and $S$ is a subset of $W$ with 
		$2\leq|S|<\frac{n+5}{4}-\delta(G)$.
		If $G[W]$ is connected and $G$ has a $1$-quasi-HIT $T(v)$ with $V(T(v))=N[u]\cup S$, where $v\in S$, then $G$ has a HIST.
	\end{lemma}
	
	\begin{proof} Assume that $T(v)$ is a $1$-quasi-HIT  of $G$ with $V(T(v))=N[u]\cup S$, where $v\in S$. Let $S':=S\backslash\{v\}$. Then $1\leq|S'|<\frac{n+1}{4}-\delta(G)$. By Lemma~\ref{le2.9}, $G[W\backslash S']$ contains at most two components.
		
		\noindent {\bf Case 1}: 
		$G[W\backslash S']$ is connected.
		
		By (\ref{e4}), we have 
		\eqnn
		{
			\delta(G[W\backslash S'])&\geq&\frac{n-1}{2}-\delta(G)-|S'|>\frac{n-1}{2}-\frac{n+1}{4}
		=	\frac{n-3}{4}\geq4\sqrt{n-3}\ge 4\sqrt{|W\backslash S'|},
		} 
		where the second last inequality follows from the condition that $n\ge 259$ 
		and the last inequality follows from 
		the fact that 
		$$
		|W\backslash S'|= n-|N[u]\cup S'|
		\le n-\delta(G)-2\le n-3.
		$$
		Then, by Theorem~\ref{Thm1.2}, $G[W\backslash S']$ has a HIST. 
		
		Note that $V(T(v))\cap (W\backslash S')=\{v\}$ and $V(T(v))\cup (W\backslash S')=V(G)$. By Lemma~\ref{le2.11}, $G$ has a HIST.

		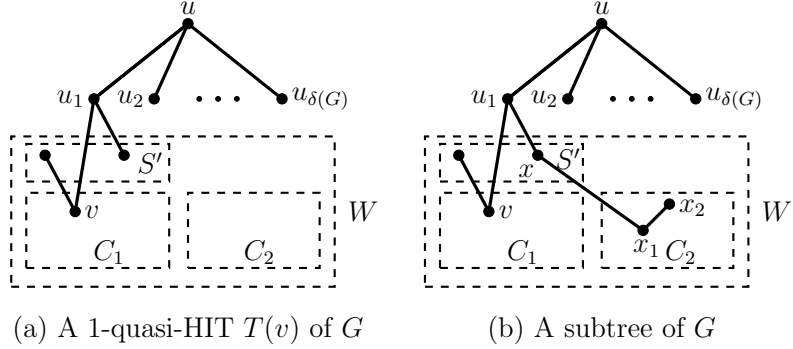
\begin{figure}[H]
			\begin{center}  
				\begin{tikzpicture}[scale=0.5]
					\tikzstyle{every node}=[font=\normalsize,scale=0.9]
					\filldraw (4.5,10)circle(0.75ex);
					\draw (4.5,10)[]node[above=0pt]{$u$};			
					\filldraw (2,8)circle(0.75ex);
					\draw (2,8)[]node[left=0pt]{$u_1$};
					\filldraw (7,8)circle(0.75ex);
					\draw (3.6,8)[]node[left=0pt]{$u_2$};			
					\filldraw (3.6,8)circle(0.75ex);
					\draw (7,8)[]node[right=0pt]{$u_{\delta(G)}$};				
					\filldraw (4.8,8)circle(0.35ex);
					\filldraw (5.3,8)circle(0.35ex);
					\filldraw (5.8,8)circle(0.35ex);
					\filldraw(2.8,6.5)circle(0.75ex);
					\filldraw(0.7,6.5)circle(0.75ex);
					\filldraw(1.5,5)circle(0.75ex);
					\draw[very thick] (4.5,10)--(2,8);		
					\draw[very thick] (4.5,10)--(3.6,8);		
					\draw[very thick] (4.5,10)--(7,8);		
					\draw[very thick] (4.5,10)--(2,8);	  		
					\draw[very thick] (2,8)--(1.5,5);			
					\draw[very thick] (2,8)--(2.8,6.5);		
					\draw[very thick] (1.5,5)--(0.7,6.5);								
					\draw[thick, dashed](8.4,7)rectangle(-0.2,3);	 				
					\draw[thick, dashed](4,5.5)rectangle(0.2,3.5);	   	
					\draw[thick, dashed](8,5.5)rectangle(4.5,3.5);		
					\draw[thick, dashed](4,5.8)rectangle(0.2,6.8);		
					\draw (2.4,4.5)[] node[below=0pt]{$C_1$};			
					\draw (6.4,4.5)[] node[below=0pt]{$C_2$};	
					\draw (2.9,6.3)[] node[right=0pt]{$S'$};	
					\draw (1.5,5) node[right=0pt]{$v$};				
					\draw (8.5,5)[] node[right=0pt]{$W$};
					\filldraw (15.5,10)circle(0.75ex);
					\draw (15.5,10)[]node[above=0pt]{$u$};	
					\filldraw (13,8)circle(0.75ex);
					\draw (13,8)[]node[left=0pt]{$u_1$};
					\filldraw (18,8)circle(0.75ex);
					\draw (14.6,8)[]node[left=0pt]{$u_2$};			
					\filldraw (14.6,8)circle(0.75ex);
					\draw (18,8)[]node[right=0pt]{$u_{\delta(G)}$};				
					\filldraw (15.8,8)circle(0.35ex);
					\filldraw (16.3,8)circle(0.35ex);
					\filldraw (16.8,8)circle(0.35ex);
					\filldraw(13.8,6.5)circle(0.75ex);
					\filldraw(16.6,4.5)circle(0.75ex);
					\filldraw(17.3,5.2)circle(0.75ex);
					\draw(16.7,4.5)[]node[below=0pt]{$x_1$};
					\draw(17.3,5.1)[]node[right=0pt]{$x_2$};
					\draw[very thick] (13.8,6.5)--(16.6,4.5);
					\draw[very thick] (16.6,4.5)--(17.3,5.2);
					\filldraw(11.7,6.5)circle(0.75ex);
					\filldraw(12.5,5)circle(0.75ex);
					\draw[very thick] (15.5,10)--(13,8);		
					\draw[very thick] (15.5,10)--(14.6,8);		
					\draw[very thick] (15.5,10)--(18,8);		
					\draw[very thick] (15.5,10)--(13,8);	  		
					\draw[very thick] (13,8)--(12.5,5);			
					\draw[very thick] (13,8)--(13.8,6.5);		
					\draw[very thick] (12.5,5)--(11.7,6.5);								
					\draw[thick, dashed](19.4,7)rectangle(10.8,3);	 				
					\draw[thick, dashed](15,5.5)rectangle(11.2,3.5);	   	
					\draw[thick, dashed](19,5.5)rectangle(15.5,3.5);		
					\draw[thick, dashed](15,5.8)rectangle(11.2,6.8);		
					\draw (13.4,4.5)[] node[below=0pt]{$C_1$};			
					\draw (17.6,4.5)[] node[below=0pt]{$C_2$};	
					\draw (14,6.4)[] node[right=0pt]{$S'$};	
					\draw (13.5,6.5)[] node[below=0pt]{$x$};	
					\draw (12.5,5) node[right=0pt]{$v$};				
					\draw (19.5,5)[] node[right=0pt]{$W$};
					\draw (4.5,2.5) node[below=0pt]{(a)~A 1-quasi-HIT $T(v)$ of $G$};	
					\draw (15.5,2.5) node[below=0pt]{(b)~A subtree of $G$};	
				\end{tikzpicture}
			\end{center}	
			\caption{Two subtrees of $G$}
			\label{f3}
		\end{figure}

		\noindent {\bf Case 2}: 
		$G[W\backslash S']$ contains exactly two components $C_1$ and $C_2$ (see Figure~\ref{f3} (a)).
		
		Since $|S'|+\delta(G)<\frac{n+1}4$, Claim A below follows directly from Lemma~\ref{le2.10} (iii). 
		
		\noindent {\bf Claim A}: For any subset $S_0$ of $C_i$, where $i\in [2]$, if $|S_0|\le 2$, then $C_i-S_0$ contains a HIST.  
		
		Clearly, $v\in C_1\cup C_2$. 
		Assume that $v\in C_1$. 
		Since $G[W]$ is connected and $|S'|\geq1$, there exists some $x\in S'$ with $|N_{C_2}(x)|\geq1$. 
		

		\noindent {\bf Subcase 2.1}: 
		$|N_{C_2}(x)|=|C_2|$.
		
		Then $T(v)$ can be extended 
		to a 1-quasi-HIT $T'(v)$ with edge set  $E(T(v))\cup E(x,C_2)$ and $d_{T'(v)}(v)=2$. Note that $C_1\cap V(T'(v))=\{v\}$ and $C_1\cup V(T'(v))=V(G)$.
		By Claim A, $C_1$ has a HIST $T_1$.  Then, $E(T'(v))\cup E(T_1)$ induces 
		a HIST of $G$ 
		by Lemma~\ref{le2.11}. 
		
		Now assume that $1\leq |N_{C_2}(x)|<|C_2|$. 
		By Lemma~\ref{le2.4}, there exist $x_1,x_2\in C_2$ such that $xx_1,x_1x_2\in E(G)$ and $xx_2\notin E(G)$ (see Figure~\ref{f3} (b)). 
		
		\noindent {\bf Subcase 2.2}: 
		$|N_{C_2}(x)|=1$.
		
By (\ref{e4}),
\eqn{le2.13-e1}
{
|N_{C_1}(x)|&=&|N_W(x)|-|N_{C_2}(x)|-|N_{S'}(x)|
			\ge 
	\left(\frac{n-1}{2}
	-\delta(G)\right)
	-1-\left(|S'|-1
	\right)
\nonumber \\
	&=&
		\frac{n-1}2-\delta(G)-|S'|
		>	\frac{n-1}{2}-\frac{n+1}{4}
		=\frac{n-3}{4}\geq2,
	}
		where the second last inequality follows from the condition that $|S'|+\delta(G)<\frac{n+1}4$.
		
		Since $	|N_{C_1}(x)|\ge 2$, 
		there exists $x_3\in  N_{C_1}(x)\backslash \{v\}$. 
		Then $G$ has a 2-quasi-HIT $T(v,x_1)$ with edge set  $E(T(v))\cup\{xx_1,xx_3,x_1x_2\}$. 
		By Claim A, 
			both $C_1-x_3$ and $C_2-x_2$ have HISTs. 
		Note that $(C_1\backslash\{x_3\})\cap V(T(v,x_1))=\{v\}$, $(C_2\backslash\{x_2\})\cap V(T(v,x_1))=\{x_1\}$ and $(C_1\backslash\{x_3\})\cup(C_2\backslash\{x_2\})\cup V(T(v,x_1))=V(G)$. By Lemma~\ref{le2.12}, $G$ has a HIST $T$. 
		
		\noindent {\bf Subcase 2.3}: 
		$2\leq |N_{C_2}(x)|<|C_2|$.
		
		Choose any vertex $x_4$ in $ N_{C_2}(x)\backslash\{x_1\}$.
		Then $x_4\neq x_2$ as $xx_2\notin E(G)$. Thus $G$ has a 2-quasi-HIT $T(v,x_1)$ with $E(T(v,x_1))=E(T(v))\cup\{xx_1,xx_4,x_1x_2\}$. 
		By Claim A, 
			both $C_1$ and $C_2-\{x_2,x_4\}$ contain HISTs. 
		Note that $C_1\cap V(T(v,x_1))=\{v\}$, $(C_2\backslash\{x_2,x_4\})\cap V(T(v,x_1))=\{x_1\}$ and $C_1\cup (C_2\backslash\{x_2,x_4\})\cup V(T(v,x_1))=V(G)$. Then by Lemma~\ref{le2.12}, $G$ has a HIST $T$.  
	\end{proof}

	
	\begin{lemma}\label{le2.14}
		Assume that $n\geq 143$,  $\delta(G)<\frac{n+1}{4}$,  and 
		$G[W]$ contains exactly two components $C_1$ and $C_2$. 
		For each $j\in [2]$ and $u_l\in N(u)$, 
			$C_j+u_l$ has a spanning tree $T$ 
			with the property $d_T(u_l)\ge \min\{2,
			|N_{C_j}(u_l)|\}$ in which 
			only $u_l$ may be of degree $2$.

	\end{lemma}

	
	\begin{proof}
		Without loss of generality,
		assume that $j=1$ and $l=1$. 
		Since $\delta(G)<\frac{n+1}4$,
		by Lemma~\ref{le2.10} (i), we have $|C_1|\ge\frac{n+1}2-\delta(G)>\frac{n+1}4>4$.
		
		If $|N_{C_1}(u_1)|=|C_1|$, then $C_1+u_1$ has a spanning tree $T$ with edge set $E(u_1,C_1)$ and $d_T(u_1)=|N_{C_1}(u_1)|$, as shown in Figure~\ref{f4} (c). Notice that 
		$T$ satisfies the above property.
		
		In the following, we assume that $1\le |N_{C_1}(u_1)|<|C_1|$.
		By Lemma~\ref{le2.4}, there exist $x_1,x_2\in C_1$ such that $u_1x_1,x_1x_2\in E(G)$ and $u_1x_2\notin E(G)$, as shown in Figures~\ref{f4} (a) and \ref{f4} (b).

		Note that $\delta(G)<\frac{n+1}4$, 
		then Claim B below follows directly from Lemma~\ref{le2.10} (iii).

		\noindent {\bf Claim B}: 
		For any subset $S_0$ of $C_1$, if $|S_0|\le 2$, $C_1-S_0$ contains a HIST.
		
		If $|N_{C_1}(u_1)|=1$, by Claim B, $C_1-x_2$ has a HIST $T_0$, implying that $C_1+u_1$ has a spanning tree $T$ with edge set $E(T_0)\cup\{x_1u_1,x_1x_2\}$ and $d_T(u_1)=|N_{C_1}(u_1)|$. Obviously, $T$ satisfies the above property.

		If $2\leq |N_{C_1}(u_1)|<|C_1|$, choose any vertex $x_3\in N_{C_1}(u_1)\backslash\{x_1\}$.
		Clearly,  $x_3\neq x_2$,
		as shown in Figure 4 (b). By Claim B, $C_1-\{x_2,x_3\}$ has a HIST $T_1$, implying that $C_1+u_1$ has a spanning tree $T$ with edge set $E(T_1)\cup\{x_1x_2,u_1x_1,u_1x_3\}$ and $d_T(u_1)=2$. Note that $T$ satisfies the above property.
	\end{proof}
	
		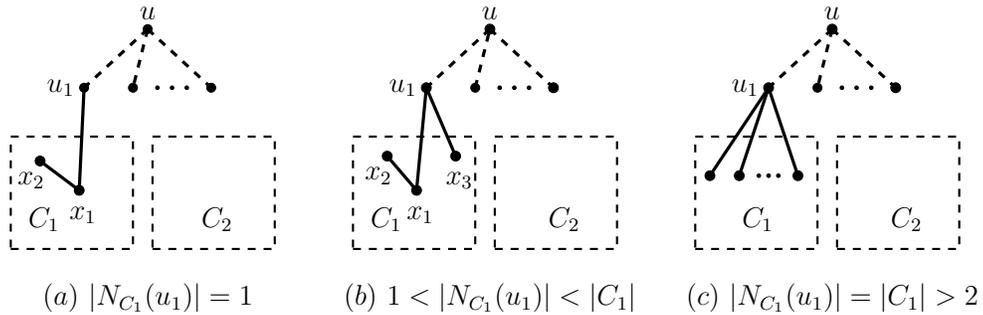
\begin{figure}[H]
		\begin{center}  
			\begin{tikzpicture}[scale=0.65]
				\tikzstyle{every node}=[font=\normalsize,scale=0.9]
				\draw (-5,9.5) node[above=0pt]
				{$u$};
				\draw (-6.3,8.3)node[left=0pt]{$u_1$};

				\draw [](-6.8,5.7) node[right=0pt]{$x_1$};
				
				\draw [](-6.9,6.4) node[left=0pt]{$x_2$};
				
				\draw (-4.1,5.6)[] node[right=0pt]{$C_2$};	
				\draw (-6.6,5.6)[] node[left=0pt]{$C_1$};	
				
				\filldraw (-5,9.5)circle(0.5ex);
				
				\filldraw (-6.3,8.3)circle(0.5ex);
				\filldraw (-5.3,8.3)circle(0.5ex);
				\filldraw (-3.7,8.3)circle(0.5ex);
				
				\filldraw (-4.2,8.3)circle(0.2ex);
				\filldraw (-4.5,8.3)circle(0.2ex);
				\filldraw (-4.8,8.3)circle(0.2ex);

				\filldraw (-6.4,6.2)circle(0.55ex);
				\filldraw (-7.2,6.8)circle(0.55ex);
				
				\draw[very thick, dashed] (-5,9.5)--(-5.3,8.3);		
				\draw[very thick, dashed] (-5,9.5)--(-6.3,8.3);		
				\draw[very thick, dashed] (-5,9.5)--(-3.7,8.3);		
				\draw[very thick] (-6.4,6.2)--(-6.3,8.3);	
				\draw[very thick] (-6.4,6.2)--(-7.2,6.8);	
				\draw[thick, dashed] (-7.8,7.3)rectangle(-5.3,5);
				
				\draw[thick, dashed] (-2.4,7.3)rectangle(-4.9,5);
				\draw (-5,4) node[]{$(a)$~$|N_{C_1}(u_1)|=1$};
				~~~~~~~~~~~~~~~~~~~~~~~~~~~~~~~~~~~~~~~~~~~~~~~~
				\draw (2,9.5) node[above=0pt]{$u$};
				\draw (0.7,8.3) node[left=0pt]{$u_1$};
				
				\draw [](0.1,5.7) node[right=0pt]{$x_1$};
				\draw [](0.9,6.4) node[right=0pt]{$x_3$};
				
				\draw [](0.2,6.5) node[left=0pt]{$x_2$};
				
				\draw (3,5.6)[] node[right=0pt]{$C_2$};	
				\draw (0.4,5.6)[] node[left=0pt]{$C_1$};	
				
				\filldraw (2,9.5)circle(0.55ex);
				
				\filldraw (0.5,6.2)circle(0.55ex);
				
				\filldraw (-0.1,6.9)circle(0.55ex);
				
				\filldraw (1.3,6.9)circle(0.55ex);
				
				\filldraw (3.3,8.3)circle(0.55ex);
				\filldraw (0.7,8.3)circle(0.55ex);
				\filldraw (1.7,8.3)circle(0.5ex);

				\filldraw (2.2,8.3)circle(0.2ex);
				\filldraw (2.5,8.3)circle(0.2ex);
				\filldraw (2.8,8.3)circle(0.2ex);

				\draw[very thick, dashed] (2,9.5)--(0.7,8.3);
				\draw[very thick, dashed] (2,9.5)--(3.3,8.3);			
				\draw[very thick, dashed](1.7,8.3)--(2,9.5);

				\draw[very thick] (1.3,6.9)--(0.7,8.3);
				\draw[very thick] (0.5,6.2)--(0.7,8.3);
				\draw[very thick] (0.5,6.2)--(-0.1,6.9);

				\draw[thick, dashed] (-0.8,7.3)rectangle(1.7,5);
				
				\draw[thick, dashed] (4.6,7.3)rectangle(2.1,5);		
				
				\draw (2,4) node[]{$(b)$~$1<|N_{C_1}(u_1)|<|C_1|$};
				~~~~~~~~~~~~~~~~~~~~~~~~~~~~~~~~~~~~~~~~~~~~~~~~~~~~~~~~~~~~~
				
				\draw (9,9.5) node[above=0pt]{$u$};
				\draw (7.7,8.3) node[left=0pt]{$u_1$};
				
				\filldraw (9,9.5)circle(0.55ex);
				
				\filldraw (10.3,8.3)circle(0.55ex);
				\filldraw (7.7,8.3)circle(0.55ex);
				\filldraw (8.7,8.3)circle(0.5ex);
				
				\filldraw (6.5,6.5)circle(0.55ex);
				\filldraw (7.1,6.5)circle(0.55ex);
				\filldraw (8.3,6.5)circle(0.55ex);
				
				\filldraw (7.5,6.5)circle(0.2ex);
				\filldraw (7.7,6.5)circle(0.2ex);
				\filldraw (7.9,6.5)circle(0.2ex);
				
				\filldraw (9.2,8.3)circle(0.2ex);
				\filldraw (9.5,8.3)circle(0.2ex);
				\filldraw (9.8,8.3)circle(0.2ex);
				
				\draw[very thick, dashed] (9,9.5)--(7.7,8.3);
				\draw[very thick, dashed] (9,9.5)--(10.3,8.3);			
				\draw[very thick, dashed] (9,9.5)--(8.7,8.3);
				
				\draw[very thick] (7.7,8.3)--(6.5,6.5);
				\draw[very thick] (7.7,8.3)--(7.1,6.5);
				\draw[very thick] (7.7,8.3)--(8.3,6.5);
				\draw[thick, dashed] (6.2,7.3)rectangle(8.7,5);
				\draw[thick, dashed] (11.6,7.3)rectangle(9.1,5);	
				\draw (9,4) node[]{$(c)$~$|N_{C_1}(u_1)|=|C_1|>2$};
				\draw (10,5.6)[] node[right=0pt]{$C_2$};	
				\draw (8,5.6)[] node[left=0pt]{$C_1$};	
			\end{tikzpicture}
		\end{center}	
		\caption{Three subtrees of $G$}
		\label{f4}
	\end{figure}
	
	\section{Proof of Theorem~\ref{Thm1.5}}	
	In this section, we will prove Theorem~\ref{Thm1.5}. 
	Let $G$ be a connected graph 
	of order $n$.
	By Theorem~\ref{Thm1.2},
	if $\delta(G)\ge 4\sqrt n$, 
	then $G$ has a HIST.
	Thus, by Corollary~\ref{cor1}, 
	in order to prove 
	Theorem~\ref{Thm1.5}, 
	it suffices to establish the following 
	conclusion. 
	
	\begin{prop}\label{prop5-1}
		Let $G$ be a connected graph of order $n\ge 270$ with 
		$\delta(G)<4\sqrt n$
		and $NC(G)\ge \frac{n-1}2$.
		Assume that 
		$G$ is not isomorphic to any $H_i$, $i\in [3]$, 
		shown in Figure~\ref{f1}, 
		nor a graph with  a pendant vertex which is adjacent to a vertex of degree $2$.
		Then, 
		$G$ contains a HIST. 
	\end{prop}

	From now on,  we assume that 
	$G$ is a graph satisfying the given 
	conditions in Proposition~\ref{prop5-1}. 
	Simplify write $\delta$ for $\delta(G)$. 
	Assume that $u$ is a vertex of $G$ with $d(u)=\delta$, $N(u)=\{u_i: i\in [\delta]\}$ and $W=V(G)\backslash  N[u]$. Then $W\neq\emptyset$ and $E(N(u),W)\neq \emptyset$ as $G$ is connected. 
	Let $U_i=N_W(u_i)$ for each $i\in [\delta]$.
	The proof of Proposition~\ref{prop5-1} will be completed in the subsections below.

\subsection{$G[W]$ is connected}

In this subsection, assume that $|U_1|=\max\limits_{i\in[\delta]}|U_i|$.
Then $|U_1|\geq 1$ as $E(N(u),W)\neq\emptyset$.  

\subsubsection{$\delta\ne2$ or $u_1u_2\in E(G)$}
Since $\delta<4\sqrt{n}$, $|W|=n-1-\delta>n-1-4\sqrt{n}>203$.
We first consider the case $|U_1|=|W|$. If $\delta\neq2$, then $G$ has a HIST with edge set $E(u,N(u))\cup E(u_1,U_1)$. Otherwise, $\delta=2$ and $u_1u_2\in E(G)$, then $G$ has a HIST with edge set  $E(u_1,N(u_1))$.

Now assume that $1\leq |U_1|<|W|$. By Lemma~\ref{le2.4}, there are two vertices $x_1,x_2\in W$ 
such that $u_1x_1,x_1x_2\in E(G)$ and $u_1x_2\notin E(G)$. 

If $|U_1|=1$, then by assumption, $\delta\geq2$ and $|U_i|\leq1$ for all $i\in[\delta]\backslash\{1\}$. By Lemma~\ref{le2.8}, $N(u)$ is a clique. Thus, $G$ has a 1-quasi-HIT $T(x_1)$ with edge set $E(u_1,N[u]\backslash\{u_1\})\cup\{u_1x_1,x_1x_2\}$ and vertex set $N[u]\cup S$, where $S=\{x_1,x_2\}$.
Observe that 
	$$
	|S|=2<\frac{n+5}{4}-\frac{n-7}{4}\leq\frac{n+5}{4}-4\sqrt{n}<\frac{n+5}{4}-\delta. 
	$$
Since $G[W]$ is connected, by Lemma~\ref{le2.13}, $G$ has a HIST.
	
If $2\leq |U_1|<|W|$, then there exists $x_3\in U_1\backslash\{x_1\}$. Clearly,   $x_3\neq x_2$. Then, $G$ has a 1-quasi-HIT $T(x_1)$ with edge set $E(u,N(u))\cup\{u_1x_1, x_1x_2, u_1x_3\}$ and vertex set $N[u]\cup S$, where $S=\{x_1,x_2,x_3\}$. Observe that 
$$
|S|=3<\frac{n+5}{4}-\frac{n-7}{4}\leq\frac{n+5}{4}-4\sqrt{n}<\frac{n+5}{4}-\delta. 
$$
Since $G[W]$ is connected, by Lemma~\ref{le2.13}, $G$ has a HIST.
	
Hence Proposition~\ref{prop5-1} holds when $G[W]$ is connected and 
either $\delta\ne 2$ or $u_1u_2\in E(G)$.

\subsubsection{$\delta=2$ and $u_1u_2\notin E(G)$}

Note that $U_1,U_2\neq\emptyset$
and  $|U_1|\geq|U_2|$.
Since $u_1u_2\notin E(G)$,
by $(\ref{Th1.5-e1})$, we have
\eqn{e12}
{|U_1|+|U_2|\ge |U_1\cup U_2|
	=|N(u_1)\cup N(u_2)|-1\geq
	\frac{n-3}2>133.
}

We first consider the case
$U_1\cap U_2\ne \emptyset$.
Let $v_1\in U_1\cap U_2$.
By (\ref{e12}), $|U_1|\ge 2$. 
Let $u'_1\in U_1\backslash \{v_1\}$.
Then $G$ has a 1-quasi-HIT $T(v_1)$ with edge set  $\{uu_1,u_1v_1,u_1u_1',u_2v_1\}$
and vertex set $N[u]\cup S'$,
where $S'=\{v_1,u_1'\}$,
as shown in Figure~\ref{f5} (a).
Observe that 
$$|S'|=2<\frac{n+5}{4}-2
=\frac{n+5}{4}-\delta.$$ 
Since $G[W]$ is connected, by Lemma~\ref{le2.13}, $G$ has a HIST.


\begin{figure}[H]
	\begin{center}  
		\begin{tikzpicture}[scale=0.76]
			\tikzstyle{every node}=[font=\normalsize,scale=0.9]
			\draw (-5,9.5) node[above=0pt]{$u$};
			\draw (-6.3,8.3)node[left=0pt]{$u_1$};
			\draw (-3.7,8.3) node[right=0pt]{$u_2$};
			\draw (-7,6.5) node[below=0pt]{$u_1'$};
			\draw (-5.5,6.35) node[below=0pt]{$v_1$};		
			\draw [](-6.6,5.6) node[right=0pt]{$S'$};
			\draw (-3.2,5.6) node[right=0pt]{$W$};			
			\filldraw (-5,9.5)circle(0.55ex);
			\filldraw (-6.3,8.3)circle(0.55ex);
			\filldraw (-3.7,8.3)circle(0.55ex);
			\filldraw (-7,6.5)circle(0.55ex);
			\filldraw (-5.6,6.5)circle(0.55ex);
			\filldraw (-4.2,6.5)circle(0.55ex);
			\filldraw (-2.6,6.5)circle(0.55ex);
			\filldraw (-3.1,6.5)circle(0.25ex);
			\filldraw (-3.4,6.5)circle(0.25ex);
			\filldraw (-3.7,6.5)circle(0.25ex);
			\draw[very thick] (-5,9.5)--(-6.3,8.3);		
			\draw[very thick,dashed] (-5,9.5)--(-3.7,8.3);		
			\draw[very thick] (-7,6.5)--(-6.3,8.3);		
			\draw[very thick] (-5.6,6.5)--(-6.3,8.3);		
			\draw[very thick] (-5.6,6.5)--(-3.7,8.3);
			\draw[thick, dashed] (-7.8,5)rectangle(-2.2,7.3);
			\draw[thick, dashed] (-7.5,5.3)rectangle(-5,7);
			\draw (-5,4.3) node[]{(a)~$v_1\in U_1\cap U_2$};
			~~~~~~~~~~~~~~~~~~~~~~~~~~~~~~~~~~~~~~~~~~~~~~~~
			\draw (3,9.5) node[above=0pt]{$u$};
			\draw (1.7,8.3) node[left=0pt]{$u_1$};
			\draw (4.3,8.3) node[right=0pt]{$u_2$};
			\draw (0.7,6.7) node[below=0pt]{$u_2'$};
			\draw (1.9,6.3) node[left=0pt]{$v_1$};	
			\draw (3.6,6.5) node[below=0pt]{$v_{k-1}$};
			\draw (4.3,6.5) node[below=0pt]{$v_k$};
			\draw (1.2,5.9) node[below=0pt]{$v_1'$};
			\draw (2.7,5.9) node[below=0pt]{$v_{k-1}'$};	
			\draw [](3.7,5.6) node[right=0pt]{$S'$};
			\draw (5.4,5.6) node[right=0pt]{$W$};
			\filldraw (3,9.5)circle(0.55ex);
			\filldraw (4.3,8.3)circle(0.55ex);
			\filldraw (1.7,8.3)circle(0.55ex);
			\filldraw (1,6.5)circle(0.55ex);
			\filldraw (1.9,6.5)circle(0.55ex);
			\filldraw (3.2,6.5)circle(0.55ex);
			\filldraw (4.1,6.5)circle(0.55ex);
			\filldraw (4.8,6.5)circle(0.55ex);
			\filldraw (6,6.5)circle(0.55ex);
			\filldraw (2.35,6.5)circle(0.25ex);
			\filldraw (2.55,6.5)circle(0.25ex);
			\filldraw (2.75,6.5)circle(0.25ex);
			\filldraw (5.2,6.5)circle(0.25ex);
			\filldraw (5.4,6.5)circle(0.25ex);
			\filldraw (5.6,6.5)circle(0.25ex);
			\filldraw (1.5,5.7)circle(0.55ex);
			\filldraw (2.8,5.7)circle(0.55ex);
			\draw[very thick] (3,9.5)--(1.7,8.3);
			\draw[very thick,dashed] (3,9.5)--(4.3,8.3);			
			\draw[very thick] (1.7,8.3)--(1,6.5);			
			\draw[very thick] (1.7,8.3)--(1.9,6.5);			
			\draw[very thick] (2.2,6.5)--(1.9,6.5); 
			\draw[very thick] (3.2,6.5)--(2.9,6.5); 			
			\draw[very thick] (3.2,6.5)--(4.1,6.5);			
			\draw[very thick] (4.1,6.5)--(4.3,8.3);			
			\draw[very thick] (3.2,6.5)--(2.8,5.7);			
			\draw[very thick] (1.9,6.5)--(1.5,5.7);	  			
			\draw[thick, dashed] (0,5)rectangle(6.3,7.3);			
			\draw[thick, dashed](0.2,5.15)rectangle(4.6,7.1);	    
			\draw (3.2,4.3) node[]{$(b)$~a shortest $(U_1,U_2)$-path $v_1v_2\cdots v_k$};
			~~~~~~~~~~~~~~~~~~~~~~~~~~~~~~~~~~~~~~~~~~~~~~~~~~~~~~~~~~~~~~
			
		\end{tikzpicture}
		
	\end{center}	
	\caption{1-quasi-HITs of $G$ when $\delta=2$ and $u_1u_2\notin E(G)$}
	
	\label{f5}
\end{figure}
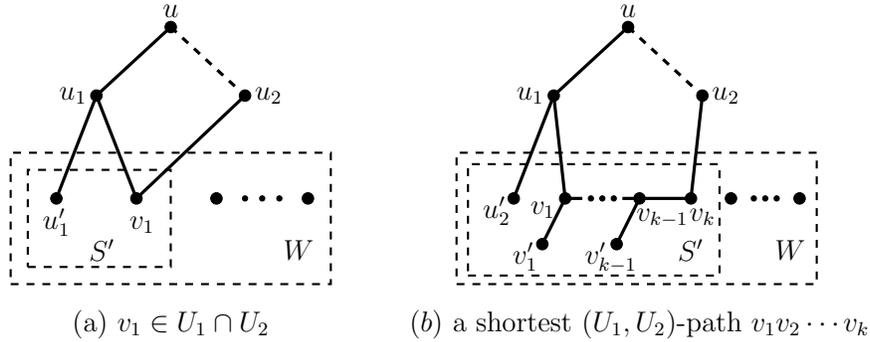

Now assume that 
$U_1\cap U_2=\emptyset$. 
Since $G[W]$ is connected, 
by Lemma~\ref{le2.5}, there is a $(U_1,U_2)$-path in $G[W]$. 
Let $P:=v_1v_2\cdots v_k$ be a shortest $(U_1,U_2)$-path in $G[W]$, where $k\geq 2$, $v_1\in U_1$
and $v_k\in U_2$. 
Clearly, 
$\{v_j: 2\le j\le k-1\}\cap 
(U_1\cup U_2)=\emptyset$. 
By (\ref{e12}), $|U_1|\ge2$, and so
there exists  $u_2'\in U_1\backslash\{v_1\}$.

We claim that $k\leq5$.
Otherwise, if $k\geq6$, then $N(v_3)\subseteq (W\backslash(U_1\cup U_2\cup P))\cup\{v_2,v_4\}$ and $N(u)=\{u_1,u_2\}$, which implies 	
\begin{eqnarray*}
	|N(u)\cup N(v_3)|&\leq&|W|-\left(|U_1|+|U_2|+k-2\right)+4
\	\le\ (n-3)
	-\left(\frac{n-3}{2}+6-2\right)+4
	\\ &=&
	\frac{n-3}{2}<\frac{n-1}{2},
\end{eqnarray*}
a contradiction to the condition of (\ref{Th1.5-e1}). Thus, $k\le 5$.

For each $v_i\in P$, 
by (\ref{e4}), we have 
$$
|N_W(v_i)|=d_W(v_i)\geq\frac{n-1}{2}-\delta=\frac{n-5}{2}\geq9,
$$ 
and thus $|N_{W-P}(v_i)|=|N_W(v_i)|-(k-1)\geq5$. So there exist $v_1'\in N_{W-P}(v_1)\backslash\{u_2'\}$ and $v_i'\in N_{W-P}(v_i)\backslash\{u_2',v_1',\ldots,v_{i-1}'\}$ for all $2\leq i\leq k-1\leq4$. Let $S'=\bigcup_{i=1}^{k-1}\{v_i,v_i'\}\cup\{u_2',v_k\}$. Then, $G$ has a 1-quasi-HIT $T(v_k)$ with 
vertex set 
$N[u]\cup S'$ and edge set 
$$
E(P)\cup\{uu_1,u_1u_2',u_1v_1,u_2v_k\}\cup
\{v_jv'_j: j=1,2,\ldots,k-1\},
$$ 
as shown in Figure~\ref{f5} (b). Note that 
$$
4\leq|S'|=2(k-1)+2=2k\leq10<\frac{n+5}{4}-2=\frac{n+5}{4}-\delta.
$$ 
Since $G[W]$ is connected, by Lemma~\ref{le2.13}, $G$ has a HIST.


Hence Proposition~\ref{prop5-1}
holds when $G[W]$ is connected, $\delta=2$ and $u_1u_2\notin E(G)$. 
Combining the conclusions in 
Subsections 5.1.1 and 5.1.2 yields that 
Proposition~\ref{prop5-1}
holds when $G[W]$ is connected.

\subsection{$G[W]$ is disconnected}

Since $\frac{n+5}{4}-\delta>\frac{n+5}4-4\sqrt{n}\geq\frac{n+5}4-\frac{n-7}4=3>0$, 
by Lemma~\ref{le2.9}, 
$G[W]$ contains at most two components. Thus,  $G[W]$ contains exactly two components, 
say $C_1$ and $C_2$. By Lemma~\ref{le2.10} (i), we have 
\eqn{e13}
{\frac{n+1}2-\delta
\leq|C_i|
\leq\frac{n-3}2.
}

We claim that $\delta\geq2$. Otherwise, if $\delta=1$, then $|W|=n-2$. By (\ref{e4}), we have
$$
\delta(G[W])\geq\frac{n-1}{2}-\delta=\frac{n-3}{2}>\frac{n-4}{2}
=\frac{|W|-2}{2}.
$$
Thus,  by Lemma~\ref{le2.6}, 
$G[W]$ is connected, a contradiction to the condition that $G[W]$ is disconnected. So in the following, we assume that $\delta\geq2$. We show that the conclusion holds 
in the following cases. 

\subsubsection{$\delta=2$}

By (\ref{e13}), we have
$|C_i|=\frac{n-3}{2}$ for both $i=1,2$.
We claim that $C_i$ is a clique for each $i\in[2]$. Otherwise, $C_i$ has two non-adjacent vertices $x$ and $y$,
implying that $|N_{N(u)}(x)\cup N_{N(u)}(y)|\geq3$ by 
Lemma~\ref{le2.10} (ii),
contradicting the fact that 
$|N(u)|=\delta=2$.
Hence each $C_i$ is a clique of 
size $\frac{n-3}2>133$. 

We first consider the case when 
there 
exists $j\in [2]$ such that 
$U_j\cap C_i\ne \emptyset$
for both $i=1,2$.
Assume that $j=1$
and 
$u_1'\in C_1\cap U_1$ and $u_1''\in C_2\cap U_1$,
as shown in Figure~\ref{f6} (a).

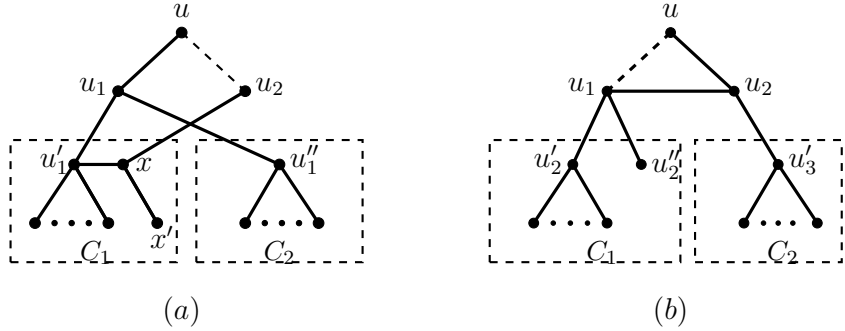
\begin{figure}[H]
	\begin{center}  
		\begin{tikzpicture}[scale=0.65]
			\tikzstyle{every node}=[font=\normalsize,scale=0.95]
			\draw (-5,9.6) node[above=0pt]{$u$};
			\draw (-6.3,8.4) node[left=0pt]{$u_1$};
			\draw (-3.7,8.4) node[right=0pt]{$u_2$};
			\draw (-7.1,6.9) node[left=0pt]{$u_1'$};
			\draw (-3,6.9) node[right=0pt]{$u_1''$};	
			\draw (-6.2,6.8) node[right=1pt]{$x$};
			\draw (-5.4,5.75) node[below=0pt]{$x'$};
			\filldraw (-5,9.5)circle(0.6ex);
			\filldraw (-6.3,8.3)circle(0.6ex);
			\filldraw (-3.7,8.3)circle(0.6ex);
			\filldraw (-3,6.8)circle(0.6ex);
			\filldraw (-6.2,6.8)circle(0.6ex);
			\filldraw (-7.2,6.8)circle(0.6ex);
			\filldraw (-8,5.6)circle(0.6ex);
			\filldraw (-6.5,5.6)circle(0.6ex);
			\filldraw (-5.5,5.6)circle(0.6ex);
			\filldraw (-3.7,5.6)circle(0.6ex);
			\filldraw (-2.2,5.6)circle(0.6ex);
			\filldraw (-7.6,5.6)circle(0.3ex);
			\filldraw (-7.25,5.6)circle(0.3ex);
			\filldraw (-6.9,5.6)circle(0.3ex);
			\filldraw (-2.6,5.6)circle(0.3ex);
			\filldraw (-2.95,5.6)circle(0.3ex);
			\filldraw (-3.3,5.6)circle(0.3ex);
			\draw[very thick] (-6.3,8.3)--(-6.3,8.3);			
			\draw[very thick] (-5,9.5)--(-6.3,8.3);
			\draw[thick,dashed] (-5,9.5)--(-3.7,8.3);
			\draw[very thick] (-6.3,8.3)--(-7.2,6.8);
			\draw[very thick] (-6.3,8.3)--(-3,6.8);
			\draw[very thick] (-3.7,8.3)--(-6.2,6.8);
			\draw[very thick] (-6.2,6.8)--(-7.2,6.8);
			\draw[very thick] (-7.2,6.8)--(-8,5.6);
			\draw[very thick] (-7.2,6.8)--(-6.5,5.6);
			\draw[very thick] (-6.2,6.8)--(-5.5,5.6);
			\draw[very thick] (-3,6.8)--(-3.7,5.6);
			\draw[very thick] (-3,6.8)--(-2.2,5.6);
			\draw[very thick] (-7.2,6.8)--(-6.5,5.6);
			\draw[very thick] (-6.2,6.8)--(-5.5,5.6);
			\draw[thick, dashed] (-8.5,4.8)rectangle(-5.1,7.3);
			\draw[thick, dashed] (-4.7,4.8)rectangle(-1.3,7.3);
			\draw[] (-3,5) node{\small$C_2$};
			\draw [](-6.75,5) node{\small$C_1$};
			\draw (-5,3.8) node[]{$(a)$};
			~~~~~~~~~~~~~~~~~~~~~~~~~~~~~~~~~~~~~~~~~~~~~~~~~~~~~~~~~~~
			\draw (5,9.6) node[above=0pt]{$u$};			
			\draw (6.3,8.4) node[right=0pt]{$u_2$};
			\draw (3.7,8.4) node[left=0pt]{$u_1$};			
			\draw (7.2,6.9) node[right=0pt]{$u_3'$};
			\draw (3,6.9) node[left=0pt]{$u_2'$};	
			\draw (4.4,6.8) node[right=0pt]{$u_2''$};	
			\filldraw (5,9.5)circle(0.55ex);
			\filldraw (6.3,8.3)circle(0.55ex);
			\filldraw (3.7,8.3)circle(0.55ex);
			\filldraw (3,6.8)circle(0.55ex);
			\filldraw (4.4,6.8)circle(0.55ex);
			\filldraw (7.2,6.8)circle(0.55ex);
			\filldraw (8,5.6)circle(0.55ex);
			\filldraw (6.5,5.6)circle(0.55ex);
			\filldraw (3.7,5.6)circle(0.55ex);
			\filldraw (2.2,5.6)circle(0.55ex);
			\filldraw (7.5,5.6)circle(0.25ex);
			\filldraw (7.2,5.6)circle(0.25ex);
			\filldraw (6.9,5.6)circle(0.25ex);
			\filldraw (2.6,5.6)circle(0.25ex);
			\filldraw (2.95,5.6)circle(0.25ex);
			\filldraw (3.3,5.6)circle(0.25ex);
			\draw[very thick] (6.3,8.3)--(3.7,8.3);			
			\draw[very thick] (4.4,6.8)--(3.7,8.3);			
			\draw[very thick] (6.5,5.6)--(7.2,6.8);			
			\draw[very thick] (5,9.5)--(6.3,8.3);
			\draw[very thick,dashed] (5,9.5)--(3.7,8.3);
			\draw[very thick] (6.3,8.3)--(7.2,6.8);
			\draw[very thick] (3.7,8.3)--(3,6.8);
			\draw[very thick] (7.2,6.8)--(8,5.6);
			\draw[very thick] (3,6.8)--(3.7,5.6);
			\draw[very thick] (3,6.8)--(2.2,5.6);
			\draw[thick, dashed] (8.5,4.8)rectangle(5.5,7.3);
			\draw[thick, dashed] (5.2,4.8)rectangle(1.3,7.3);
			\draw[] (3.6,5) node{\small$C_1$};
			\draw [] (7.3,5) node{\small$C_2$};
			\draw (5,3.8) node[]{$(b)$};
		\end{tikzpicture}
	\end{center}	
	
	\caption{Two HISTs of $G$ when $\delta=2$}
	
	\label{f6}
\end{figure}

Let $x\in N(u_2)\backslash\{u\}$. 
If $x\in\{u_1,u_1',u_1''\}$, then $G$ has a HIST $T$ with edge set 
$$
E(u_1',C_1\backslash\{u_1'\})\cup E(u_1'',C_2\backslash\{u_1'\})\cup\{uu_1,u_1u_1',u_1u_1'',xu_2\}.
$$
So we assume $x\notin\{u_1,u_1',u_1''\}$.
Then $x\in C_1\backslash\{u_1'\}$ or $x\in C_2\backslash\{u_1''\}$. By symmetry, we only need to consider $x\in C_1\backslash\{u_1'\}$. Then there exists $x'\in N_{C_1}(x)\backslash\{u_1'\}$, and thus $G$ has a HIST $T$ with edge set 
$$
E(u_1',C_1\backslash\{x',u_1'\})\cup E(u_1'',C_2\backslash\{u_1''\})\cup\{uu_1,u_1u_1',u_1u_1'',xu_2,xx'\},
$$
 as shown in Figure~\ref{f6} (a).


Next we consider the case that 
for each $j\in [2]$, 
$U_j\subseteq C_i$ 
for some $i\in [2]$. 
Assume that $U_1\subseteq C_1$ and $U_2\subseteq C_2$. Then $|U_i|\geq1$ for $i\in[2]$.

We now claim that $u_1u_2\in E(G)$.
Suppose that 
$u_1u_2\notin E(G)$.
We first show that $C_i+u_i$ is a clique for each $i\in[2]$. 
Suppose that 
$C_i+u_i$ is not a clique. 
Since $C_i$ is a clique, 
$xu_i\notin E(G)$ for some  $x\in C_i$.
Obviously, $N(x)\cup N(u_i)\subseteq (C_i\backslash\{x\})\cup\{u\}$, 
implying that 
$$
|N(x)\cup N(u_i)|\leq|C_i|-1+1=\frac{n-3}{2}
<\frac{n-1}{2},
$$ 
a contradiction 
to the condition that $NC(G)\ge \frac{n-1}2$. 
Hence  $C_i+u_i$ is a clique of size $\frac{n-1}{2}$ for both $i=1,2$.
But, it follows that $G\cong H_1$,
a contradiction to the assumption too.

Now we know that $u_1u_2\in E(G)$. Since $G\not\cong H_2$, it is impossible that 
$|U_1|=1$ and $|U_2|=1$.
Thus, $|U_1|\geq2$ as $|U_1|\geq |U_2|$, say $|N_{C_1}(u_1)|\geq2$.
Let $u_2',u_2''\in U_1$ and $u_3'\in U_2$.
Then $G$ has a HIST $T$ with edge set 
$$
E(u_2',C_1\backslash\{u_2',u_2''\})\cup E(u_3',C_2\backslash\{u_3'\})\cup\{uu_2,u_1u_2,u_1u_2',u_1u_2'',u_2u_3'\},
$$
as shown in Figure~\ref{f6} (b).

Hence Proposition~\ref{prop5-1}
holds for the case $G[W]$ is disconnected and $\delta=2$.

\subsubsection{$\delta\ge 3$}

We show that 
the conclusion holds 
in the following case. 

\vskip .5cm
\noindent {\bf Case 1}: 
There exists $j\in [\delta]$
such that $U_j\cap C_i\ne \emptyset$ 
for both $i=1,2$. Say $j=1$.
\vskip .5cm

By Lemma~\ref{le2.14}, for each $i\in [2]$, $C_i+u_1$ has a spanning tree $T_i$ in which only vertex $u_1$ may be of degree $2$. Thus, $G$ has a HIST $T$ with edge set $E(u,N(u))\cup E(T_1)\cup E(T_2)$.

\vskip .5cm
\noindent {\bf Case 2}: 
For each $j\in[\delta]$, there exists $i\in [2]$ such that $U_j\subseteq C_i$. 
\vskip .5cm

For $i\in [2]$, let $Q_i$ be the set of 
$j\in [\delta]$ with $U_j\subseteq C_i$ and $U_j\ne \emptyset$. 
Since $G$ is connected,
$Q_i\ne \emptyset$
for each $i\in [2]$. 
We may assume that $i\in Q_i$ for 
each $i\in [2]$.  

If $|U_{j_1}|\ge 2$ and $|U_{j_2}|\ge 2$,
where $j_i\in Q_i$ for $i=1,2$, 
say $|U_1|\ge 2$ and $|U_2|\ge 2$, 
then 
by Lemma~\ref{le2.14}, 
for each $i\in [2]$, 
$C_i+u_i$ has
a spanning tree $T_i$ with $d_{T_i}(u_i)\ge 2$ 
and the property that 
only vertex $u_i$ may be of degree $2$.
Thus,  $G$ has a HIST $T$ with
edge set $E(u,N(u))\cup E(T_1)\cup E(T_2)$.

If $|U_{j_1}|=1$ and $|U_{j_2}|=1$,
where $j_i\in Q_i$ for $i=1,2$, 
say $|U_1|=1$ and $|U_2|=1$,
$|U_i|\leq1$ for all $i\in[\delta]$, then by Lemma~\ref{le2.8}, $N(u)$ is a clique. By Lemma~\ref{le2.14}, 
for each $i\in[2]$, 
$C_i+u_i$ has a HIST $T_i$ with $d_{T_i}(u_i)=1$.
Thus, $G$ has a HIST $T$ with
edge set $E(u_1,N[u]\backslash\{u_1,u_{\delta}\})\cup E(T_1)\cup E(T_2)\cup\{u_2u_{\delta}\}$.

Now it remains to consider the case
$|U_{j_1}|\ge 2$ for each $j_1\in Q_1$
and $|U_{j_2}|=1$ for each $j_2\in Q_2$. 
Particularly, 
$|U_1|\ge 2$ and $|U_2|=1$.
Since $\delta\geq3$, $u_j\in N(u_2)$ for some 
$j\in [\delta]\backslash \{2\}$.
By Lemma~\ref{le2.14}, 
for each $i\in [2]$, $C_i+u_i$ has a spanning tree $T_i$ 
with $d_{T_i}(u_i)\ge 3-i$
and the property that only $u_i$
may be of degree $2$ in $T_i$.
If $\delta\ge 4$,
then 
$G$ has a HIST $T$ with edge set $E(u,N(u)\backslash\{u_j\})\cup E(T_1)\cup E(T_2)\cup\{u_ju_2\}$.

Now assume that  $\delta=3$. 
By (\ref{e13}),  $\frac{n-5}{2}\leq|C_i|
\leq\frac{n-3}{2}$ for each $i\in[2]$.
We claim that each $C_i$ is a clique. Otherwise, by Lemma~\ref{le2.10} (ii), there exist two non-adjacent vertices $x,y\in C_i$ such that  $|N_{N(u)}(x)\cup N_{N(u)}(y)|\geq 3$.
However, 
$N_{N(u)}(x)\cup N_{N(u)}(y)\subseteq\{u_i,u_3\}$, a contradiction. 
Hence $C_i$ is a clique with 
$\frac{n-5}{2}\leq|C_i|
\leq\frac{n-3}{2}$.

\begin{figure}[H]	
	\begin{center}  
		\begin{tikzpicture}[scale=0.76]
			\tikzstyle{every node}=[font=\normalsize,scale=0.9]
			\draw (-5,9.5) node[above=0pt]{$u$};
			\draw (-6.3,8.3)node[left=0pt]{$u_1$};
			\draw (-5,8.3) node[right=0pt]{$u_2$};
			\draw (-3.7,8.3) node[right=0pt]{$u_3$};
			\draw (-6.05,6.7) node[right=0pt]{$y_1(y)$};		
			\draw (-7.8,6.7) node[right=0pt]{$y_2$};			
			\draw (-3.9,6.7) node[right=0pt]{$y'$};			
			\filldraw (-5,9.5)circle(0.55ex);
			\filldraw (-6.3,8.3)circle(0.55ex);
			\filldraw (-3.7,8.3)circle(0.55ex);
			\filldraw (-5,8.3)circle(0.55ex);
			\filldraw (-7,6.7)circle(0.55ex);
			\filldraw (-6,6.7)circle(0.55ex);
			\filldraw (-3.9,6.7)circle(0.55ex);
			\filldraw (-5.5,5.7)circle(0.55ex);
			\filldraw (-6.5,5.7)circle(0.55ex);
			\filldraw (-3.4,5.7)circle(0.55ex);
			\filldraw (-4.4,5.7)circle(0.55ex);
			\filldraw (-6.2,5.7)circle(0.25ex);
			\filldraw (-6,5.7)circle(0.25ex);
			\filldraw (-5.8,5.7)circle(0.25ex);
			\filldraw (-3.7,5.7)circle(0.25ex);
			\filldraw (-3.9,5.7)circle(0.25ex);
			\filldraw (-4.1,5.7)circle(0.25ex);
			\draw[very thick,dashed] (-5,9.5)--(-6.3,8.3);
			\draw[very thick,dashed] (-5,9.5)--(-3.7,8.3);			
			\draw[very thick] (-6.3,8.3)--(-5,8.3);
			\draw[very thick] (-5,9.5)--(-5,8.3);
			\draw[very thick] (-7,6.7)--(-6.3,8.3);
			\draw[very thick] (-6,6.7)--(-6.3,8.3);
			\draw[very thick] (-6,6.7)--(-3.7,8.3);
			\draw[very thick] (-3.9,6.7)--(-5,8.3);			
			\draw[very thick] (-6,6.7)--(-6.5,5.7);
			\draw[very thick] (-6,6.7)--(-5.5,5.7);			
			\draw[very thick] (-3.9,6.7)--(-4.4,5.7);
			\draw[very thick] (-3.9,6.7)--(-3.4,5.7);			
			\draw[thick, dashed] (-7.7,5)rectangle(-4.7,7.3);
			\draw[thick, dashed] (-4.55,5)rectangle(-2.3,7.3);
			\draw[] (-2.8,5.4) node{\small$C_2$};
			\draw [](-7.1,5.4) node{\small$C_1$};
			\draw (-5,4.3) node[]{$(a)$};
			~~~~~~~~~~~~~~~~~~~~~~~~~~~~~~~~~~~~~~~~~~~~~~~~~~~~~		
			\draw (3,9.5) node[above=0pt]{$u$};
			\draw (1.7,8.3) node[left=0pt]{$u_1$};
			\draw (3,8.3) node[right=0pt]{$u_2$};
			\draw (4.3,8.3) node[right=0pt]{$u_3$};	
			\draw (0.7,6.7) node[left=0pt]{$y_2$};
			\draw (2.1,6.5) node[above=1pt]{$y_1$};
			\draw (2.7,6.7) node[right=0pt]{$y$};
			\draw (1.1,5.7) node[below=0pt]{$y_3$};
			\draw (4.7,6.7) node[right=0pt]{$y'$};			
			\filldraw (3,9.5)circle(0.55ex);
			\filldraw (1.7,8.3)circle(0.55ex);
			\filldraw (4.3,8.3)circle(0.55ex);
			\filldraw (3,8.3)circle(0.55ex);
			\filldraw (0.7,6.7)circle(0.55ex);
			\filldraw (1.7,6.7)circle(0.55ex);
			\filldraw (2.7,6.7)circle(0.55ex);
			\filldraw (4.7,6.7)circle(0.55ex);
			\filldraw (2.2,5.7)circle(0.55ex);
			\filldraw (3.3,5.7)circle(0.55ex);
			\filldraw (1.1,5.7)circle(0.55ex);
			\filldraw (4.2,5.7)circle(0.55ex);
			\filldraw (5.2,5.7)circle(0.55ex);
			\filldraw (2.5,5.7)circle(0.25ex);
			\filldraw (2.7,5.7)circle(0.25ex);
			\filldraw (2.9,5.7)circle(0.25ex);
			\filldraw (4.5,5.7)circle(0.25ex);
			\filldraw (4.7,5.7)circle(0.25ex);
			\filldraw (4.9,5.7)circle(0.25ex);
			\draw[very thick,dashed] (3,9.5)--(1.7,8.3);
			\draw[very thick,dashed] (3,9.5)--(4.3,8.3);
			\draw[very thick] (3,9.5)--(3,8.3);
			\draw[very thick] (0.7,6.7)--(1.7,8.3);
			\draw[very thick] (1.7,6.7)--(1.7,8.3);
			\draw[very thick] (4.7,6.7)--(3,8.3);
			\draw[very thick] (2.7,6.7)--(4.3,8.3);			
			\draw[very thick] (3,8.3)--(1.7,8.3);		
			\draw[very thick] (1.7,6.7)--(1.1,5.7);
			\draw[very thick] (1.7,6.7)--(2.7,6.7);
			\draw[very thick] (2.7,6.7)--(2.2,5.7);
			\draw[very thick] (2.7,6.7)--(3.3,5.7);		
			\draw[very thick] (4.7,6.7)--(4.2,5.7);
			\draw[very thick] (4.7,6.7)--(5.2,5.7);		
			\draw[thick, dashed] (-0.2,5)rectangle(3.5,7.3);
			\draw[thick, dashed] (3.8,5)rectangle(6.1,7.3);
			\draw[] (5.6,5.4) node{\small$C_2$};
			\draw [](0.3,5.4) node{\small$C_1$};
			\draw (3,4.3) node[]{$(b)$};
		\end{tikzpicture}
	\end{center}	
	\caption{Two HISTs of $G$ when $\delta=3$ and $u_1u_2\in E(G)$}
	
	\label{f7}
\end{figure}
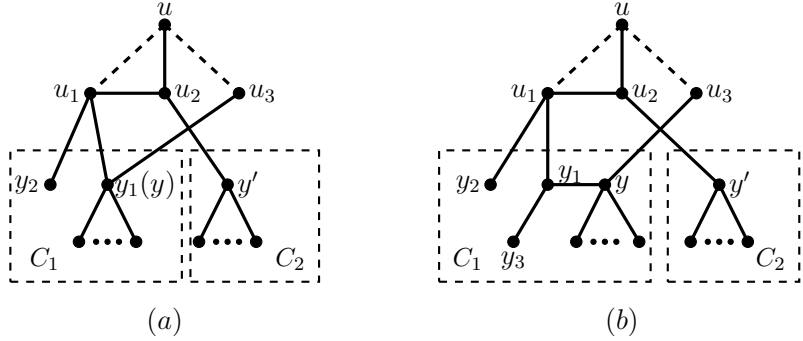

\vskip .5cm
\noindent {\bf Case 2.1}: 
$|U_1|\ge 2$, $|U_2|=1$, 
$\delta=3$ and $u_1u_2\in E(G)$.
\vskip .5cm

Assume that   $y_1,y_2\in U_1$ and  $y'\in U_2$. 
Since $d(u_3)\geq\delta=3$,
there exists $y\in N(u_3)\backslash\{u\}$.
If $y\in\{u_1,u_2,y_i,y'\}$,  where $i\in [2]$, 
then $G$ has a HIST $T$ with edge set 
$$
E(y_i,C_1\backslash\{y_i,y_{3-i}\})\cup E(y',C_2\backslash\{y'\})\cup\{uu_2,u_1u_2,u_1y_1,u_1y_2,u_2y',u_3y\},
$$
as shown in Figure~\ref{f7} (a). 

Now we assume that  $y\notin\{u_1,u_2,y_1,y_2,y'\}$ and
thus $N_{N(u)}(u_3)=\emptyset$.
Since $\delta=3$, $|U_3|\ge 2$, 
implying that $U_3\subseteq C_1$.
Note that $y\notin\{y_1,y_2\}$.
Since $C_1$ is a clique of size at least 
$\frac{n-5}2$ and $n\ge 270$,
there exists  $y_3\in N_{C_1}(y_1)\backslash\{y,y_2\}$. Then $G$ has a HIST $T$ with edge set  
$$
E(y,C_1\backslash\{y,y_2,y_3\})\cup E(y',C_2\backslash\{y'\})\cup\{uu_2,u_1u_2,u_1y_1,u_1y_2, u_2y',u_3y,y_1y_3\},
$$
as shown in Figure~\ref{f7} (b).

\vskip .5cm
\noindent {\bf Case 2.2}: 
$|U_1|\ge 2$, $|U_2|=1$, 
$\delta=3$ and $u_1u_2\notin E(G)$.
\vskip .5cm

\vskip.2cm 
If $u_2u_3\notin E(G)$, then 
$N(u_2)\subseteq \{u\}\cup U_2$, 
implying that $d(u_2)<3$, 
a contradiction to the condition 
that $\delta=3$. Thus $u_2u_3\in E(G)$.

It is known that $|C_1|\ge \frac{n-5}2$. 
It can be shown that 
$|C_1|=\frac{n-5}{2}$. 
Otherwise, $|C_1|\geq\frac{n-4}{2}$,
implying that 
$$
|C_2|=n-(\delta+1)-|C_1|
\le n-4-\frac{n-4}{2}=\frac{n-4}{2}. 
$$
Since $|U_2|=1$, there exists 
$x\in C_2\backslash N(u_2)$. 
Note that 
$N(x)\cup N(u_2)\subseteq(C_2\backslash\{x\})\cup\{u,u_3\}$, implying that 
$$
|N(x)\cup N(u_2)|\leq|C_2|-1+2
\leq\frac{n-4}{2}+1<\frac{n-1}{2},
$$ 
a contradiction to the condition of (\ref{Th1.5-e1}). Therefore, $|C_1|=\frac{n-5}{2}$ and $|C_2|=\frac{n-3}{2}$.

Next, it can be shown that $C'_1:=C_1+u_1$  is a clique.
Otherwise, as $C_1$ is a clique, 
there exists $y\in C_1\backslash N(u_1)$.
Note that $N(u_1)\cup N(y)\subseteq(C_1\backslash\{y\})\cup\{u,u_3\}$, 
implying that  
$$
|N(u_1)\cup N(y)|\leq|C_1|-1+2=\frac{n-5}{2}+1<\frac{n-1}{2},
$$ 
a contradiction to the condition of (\ref{Th1.5-e1}). 
Note that  $|C_1'|=|C_1|+1=\frac{n-3}2$.

Since $d(u_3)\geq\delta=3$, $u_3$ has at least one neighbor in $C_1'\cup C_2$. Note that $u_3$ cannot have neighbors in both $C_1$ and $C_2$.

If there exists $x_1,x_2\in N(u_3)\cap C'_1$, then $G$ has a HIST with edge set  
$$
E(x_1,C_1'\backslash\{x_1,x_2\})\cup E(x_3,C_2\backslash\{x_3\})\cup\{uu_2,u_2u_3,u_2x_3,u_3x_1,u_3x_2\},
$$
where $x_3$ is the only vertex in $U_2$, as shown in  Figure~\ref{f8} (a).
	
If $|N(u_3)\cap (C'_1\cup C_2)|\ge 2$ but $|N(u_3)\cap C'_1|\le 1$, then $N(u_3)\cap (C'_1\cup C_2)=\{u_1, z\}$, where $z\in C_2$. Observe that $G$ has a HIST with edge set $E(u_1, C_1)\cup E(z,C_2\backslash\{z\})\cup\{uu_1, u_1u_3, u_3u_2,u_3z\}$, as shown in  Figure~\ref{f8} (b).
	
If $|N(u_3)\cap (C'_1\cup C_2)|=1$, then the only vertex in $N(u_3)\cap (C'_1\cup C_2)$ is either a vertex contained in $\{u_1, x_3\}$, or a vertex contained in $(C_1'\cup C_2)\backslash\{u_1,x_3\}$, where $x_3$ is the only vertex in $U_2$. Then, $G$ is $H_3$, 
as shown in  Figure~\ref{f8} (c).

Hence Proposition~\ref{prop5-1}
holds for the case $G[W]$ is disconnected and $\delta\geq3$.

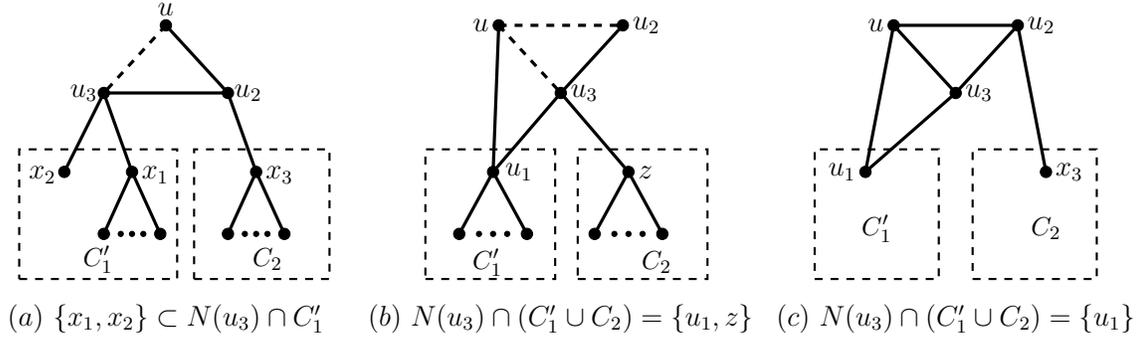
\begin{figure}[H]
	\begin{center}  
		\begin{tikzpicture}[scale=0.75]
			\tikzstyle{every node}=[font=\normalsize,scale=0.9]
			\draw (-4,9.5) node[above=0pt]{$u$};
			\draw (-5,8.3)node[left=1pt]{$u_3$};
			\draw (-3,8.3) node[right=1pt]{$u_2$};
			\draw (-4.6,6.85) node[right=0pt]{$x_1$};		
			\draw (-6.6,6.85) node[right=0pt]{$x_2$};			
			\draw (-2.4,6.85) node[right=0pt]{$x_3$};				
			\filldraw (-4,9.5)circle(0.55ex);
			\filldraw (-5.1,8.3)circle(0.55ex);
			\filldraw (-2.9,8.3)circle(0.55ex);
			\filldraw (-5.8,6.9)circle(0.55ex);
			\filldraw (-4.6,6.9)circle(0.55ex);
			\filldraw (-2.4,6.9)circle(0.55ex);
			\filldraw (-4.1,5.8)circle(0.55ex);
			\filldraw (-5.1,5.8)circle(0.55ex);
			\filldraw (-1.9,5.8)circle(0.55ex);
			\filldraw (-2.9,5.8)circle(0.55ex);
			\filldraw (-4.8,5.8)circle(0.25ex);
			\filldraw (-4.6,5.8)circle(0.25ex);
			\filldraw (-4.4,5.8)circle(0.25ex);
			\filldraw (-2.2,5.8)circle(0.25ex);
			\filldraw (-2.4,5.8)circle(0.25ex);
			\filldraw (-2.6,5.8)circle(0.25ex);
			\draw[very thick,dashed] (-4,9.5)--(-5.1,8.3);
			\draw[very thick] (-4,9.5)--(-2.9,8.3);			
			\draw[very thick] (-5.1,8.3)--(-2.9,8.3);
			\draw[very thick] (-2.4,6.9)--(-2.9,8.3);			
			\draw[very thick] (-5.8,6.9)--(-5.1,8.3);
			\draw[very thick] (-4.6,6.9)--(-5.1,8.3);			
			\draw[very thick] (-4.6,6.9)--(-5.1,5.8);
			\draw[very thick] (-4.6,6.9)--(-4.1,5.8);			
			\draw[very thick] (-2.4,6.9)--(-2.9,5.8);
			\draw[very thick] (-2.4,6.9)--(-1.9,5.8);
			\draw[thick, dashed] (-6.6,5)rectangle(-3.8,7.3);
			\draw[thick, dashed] (-3.5,5)rectangle(-1.1,7.3);
			\draw[] (-2.2,5.35) node{\small$C_2$};
			\draw [](-5.2,5.35) node{\small$C_1'$};
			\draw (-4,4.3) node[]{$(a)$~$\{x_1,x_2\}\subset N(u_3)\cap C'_1$};
			~~~~~~~~~~~~~~~~~~~~~~~~~~~~~~~~~~~~~~~~~~~~~~~~~~~~~		
			\draw (1.9,9.5) node[left=0pt]{$u$};
			\draw (4.1,9.5) node[right=0pt]{$u_2$};		
			\draw (3,8.3) node[right=0pt]{$u_3$};			
			\draw (1.8,6.9) node[right=1pt]{$u_1$};
			\draw (4.2,6.9) node[right=0pt]{$z$};						
			\filldraw (3,8.3)circle(0.55ex);
			\filldraw (1.9,9.5)circle(0.55ex);
			\filldraw (4.1,9.5)circle(0.55ex);
			\filldraw (1.8,6.9)circle(0.55ex);
			\filldraw (4.2,6.9)circle(0.55ex);
			\filldraw (2.4,5.8)circle(0.55ex);
			\filldraw (1.2,5.8)circle(0.55ex);
			\filldraw (4.8,5.8)circle(0.55ex);
			\filldraw (3.6,5.8)circle(0.55ex);
			\filldraw (1.55,5.8)circle(0.25ex);
			\filldraw (1.8,5.8)circle(0.25ex);
			\filldraw (2.05,5.8)circle(0.25ex);
			\filldraw (4.45,5.8)circle(0.25ex);
			\filldraw (4.2,5.8)circle(0.25ex);
			\filldraw (3.95,5.8)circle(0.25ex);
			
			\draw[very thick,dashed] (1.9,9.5)--(4.1,9.5);
			\draw[very thick,dashed] (1.9,9.5)--(3,8.3);
			\draw[very thick] (4.1,9.5)--(3,8.3);
			\draw[very thick] (1.8,6.9)--(3,8.3);
			\draw[very thick] (4.2,6.9)--(3,8.3);
			\draw[very thick] (1.8,6.9)--(1.2,5.8);
			\draw[very thick] (1.8,6.9)--(2.4,5.8);
			\draw[very thick] (4.2,6.9)--(3.6,5.8);
			\draw[very thick] (4.2,6.9)--(4.8,5.8);			
			\draw[very thick] (1.9,9.5)--(1.8,6.9);	
			\draw[thick, dashed] (0.6,5)rectangle(2.9,7.3);
			\draw[thick, dashed] (3.3,5)rectangle(5.6,7.3);
			\draw[] (4.7,5.3) node{\small$C_2$};
			\draw [](1.7,5.3) node{\small$C_1'$};
			\draw (3,4.3) node[]{$(b)$~$N(u_3)\cap (C'_1\cup C_2)=\{u_1, z\}$};
			\draw (8.9,9.5) node[left=0pt]{$u$};
			\draw (11.1,9.5) node[right=0pt]{$u_2$};  
			\draw (8.4,6.9) node[left =0pt]{$u_1$}; 
			\draw (11.6,6.9) node[right =0pt]{$x_3$}; 
			\draw (10,8.3) node[right=0pt]{$u_3$};
			\filldraw (10,8.3)circle(0.55ex);
			\filldraw (8.9,9.5)circle(0.55ex);
			\filldraw (11.1,9.5)circle(0.55ex);
			\filldraw (8.4,6.9)circle(0.55ex);
			\filldraw (11.6,6.9)circle(0.55ex);
			\draw[very thick] (8.9,9.5)--(11.1,9.5);
			\draw[very thick] (8.9,9.5)--(10,8.3);
			\draw[very thick] (11.1,9.5)--(10,8.3);
			\draw[very thick] (8.9,9.5)--(8.4,6.9);
			\draw[very thick] (11.1,9.5)--(11.6,6.9);
			\draw[very thick] (10,8.3)--(8.4,6.9);
			\draw[thick,dashed] (7.5,5)rectangle(9.7,7.3);
			\draw[thick,dashed] (10.3,5)rectangle(12.5,7.3);
			\draw[] (11.6,5.9) node{\small$C_2$};
			\draw [](8.6,5.9) node{\small$C_1'$};
			\draw (10,4.3) node[]{$(c)$~$N(u_3)\cap (C'_1\cup C_2)=\{u_1\}$};
		\end{tikzpicture}
	\end{center}	
	\caption{Three cases when $N(u_3)\cap (C'_1\cup C_2)\neq\emptyset$}
	
	\label{f8}
\end{figure}

\medskip 

This completes the proof of Theorem~\ref{Thm1.5}.\q

\vskip .2cm

    \section{Remarks}
	
    We do not know whether the conclusion of Theorem~\ref{Thm1.5} holds 
    when $n<270$,
	although it indeed holds for some graphs. For example, if $G$ is the Petersen graph, then 
	$\sigma(G)=6<n-1$ and $NC(G)=5\ge \frac{n-1}2$, where $n=|V(G)|=10$. It can be verified that the Petersen graph contains HISTs. 
	Thus, it is natural to propose 
 the following problem:

	\begin{prob} 
Replace
	the condition $n\ge 270$
	in Theorem~\ref{Thm1.5} 
	by a weaker one.
	\end{prob}

	\section*{Declaration of Competing Interest}
	The authors declare that they have no known competing financial interests or personal relationships that could have appeared to influence the work reported in this paper.

\end{document}